\documentclass[twoside,11pt]{amsart}

\usepackage[cmtip,all]{xy}
\usepackage{amssymb, xspace, enumerate, times, color}
\usepackage{hyperref}
\usepackage{fullpage, appendix}

\usepackage{graphicx}

\hypersetup{
    colorlinks=true, 
    linkcolor=blue, 
    urlcolor=red, 
    linktoc=all 
}

\def\sqr#1#2{{\vcenter{\hrule height.#2pt
        \hbox{\vrule width.#2pt height#1pt \kern#1pt
                \vrule width.#2pt}
        \hrule height.#2pt}}}

\theoremstyle{theorem}
\newtheorem{Theorem}{Theorem}[section]

\newtheorem{Claim}{Claim}[Theorem]

\newtheorem{Lemma}[Theorem]{Lemma}

\newtheorem{Corollary}[Theorem]{Corollary}

\newtheorem{Proposition}[Theorem]{Proposition}

\theoremstyle{definition}
\newtheorem{Set-up}[Theorem]{Set-up}
\newtheorem{Conjecture}[Theorem]{Conjecture}

\newtheorem{Assumptions and Discussion}[Theorem]{Assumptions and Discussion}

\newtheorem{Remark}[Theorem]{Remark}
\newtheorem{Example}[Theorem]{Example}
\newtheorem{Definition}[Theorem]{Definition}

\newtheorem{Problem}[Theorem]{Problem}
\newtheorem*{acknowledgement}{Acknowledgements}

\DeclareMathOperator{\Spec}{Spec}
\DeclareMathOperator{\Proj}{Proj}

\def\m{{\mathfrak m}}

\def\q{{\mathfrak q}}
\def\p{{\mathfrak p}}
\def\pp{{\mathfrak p}}

\newcommand{\ol}[1]{\overline{#1}}
\newcommand{\ovl}[1]{\overline{#1}}
\newcommand{\ul}[1]{\underline{#1}}

\def\ZZ{{\mathbb Z}}
\def\NN{{\mathbb N}}
\def\CC{{\mathbb C}}

\def\AA{{\mathbb A}}
\def\t{\mathbf{t}}

\newcommand{\F}{\mathcal{F}}
\newcommand{\OO}{\mathcal{O}}

\newcommand{\II}{\mathcal{I}}

\newcommand{\PP}{\mathbb{P}}

\def\Llra{\Longleftrightarrow}

\def\lra{\longrightarrow}

\newcommand{\be}{\begin{equation*}}
\newcommand{\ee}{\end{equation*}}
\newcommand{\bee}{\begin{equation}}
\newcommand{\eee}{\end{equation}}

\def\rk{{\rm rk}}
\def\h{{\rm ht}}

\def\Ass{{\rm Ass}}

\def\Max{{\rm Max}}
\def\Spec{{\rm Spec}}

\def\depth{{\rm depth}}

\def\sat{{\rm sat}}

\definecolor{lighterorange}{cmyk}{0,0.42,0.66,0.0}

\numberwithin{equation}{section}

\begin{document}

\setlength{\baselineskip}{1.2em}


\title{The Alexander--Hirschowitz theorem and related problems}

\author{Huy T\`ai H\`a}
\address{Department of Mathematics \\
	Tulane University \\
	6823 St. Charles Avenue \\
	New Orleans, LA 70118}
\email{tha@tulane.edu}

\author{Paolo Mantero}
\address{Department of Mathematical Sciences \\
University of Arkansas \\
Fayetteville, AR 72701}
\email{pmantero@uark.edu}

\maketitle

\begin{abstract}
We present a proof of a celebrated theorem of Alexander and Hirschowitz determining when a general set of double points in $\PP^n$ has the expected Hilbert function.
Our intended audience are Commutative Algebraists who may be new to interpolation problems. In particular, the main aim of our presentation is to provide a self-contained proof containing all details (including some we could not find in the literature). Also, considering our intended audience, we have added (a) short appendices to make this survey more accessible and (b) a few open problems related to the Alexander--Hirschowitz theorem and the interpolation problems.

\end{abstract}

\begin{center}
	\textsl{Dedicated to David Eisenbud, on the occasion of his 75th birthday.}
\end{center}

\tableofcontents

\section{Introduction: the Alexander--Hirschowitz Theorem} \label{sec.introduction}

The \emph{polynomial interpolation problem} originates from the simple fact that a polynomial in one variable over $\CC$ is completely determined by its zeros. In fact, given $r \le d$ distinct points $x_1, \dots, x_r$ on the affine line $\AA^1_\CC$ and positive integers $m_1, \dots, m_r$ such that $m_1+\ldots + m_r=d+1$, a polynomial $f(x) = a_0 + a_1x + \dots + a_dx^d$ of degree $d$ is uniquely determined by the following $(d+1)$ vanishing conditions on its derivatives, namely $f^{(j)}(x_i) = 0$ for all $i = 1, \dots, r$ and $j = 0, \dots, m_i-1$. Equivalently, the matrix arising from these vanishing conditions, which determines the parameters $a_0, \dots, a_d$, has \emph{maximal rank}. A natural question, that has been studied for a long time, is: \emph{what happens in higher dimension, meaning for polynomials in several variables?}

The problem is much more difficult for several variables, even when the multiplicities $m_1, \dots, m_r$ are all equal and the ambient space is a projective space over the complex numbers. The aim of this paper is to explore a fundamental result due to Alexander and Hirschowitz, obtained in a series of papers \cite{Hi, Al, AH1, AH2, AH3} (and simplification to its proof given by Chandler in \cite{Ch00, Ch02}),  which shows that, if $m_1 = \dots = m_r = 2$ and the points are chosen to be \emph{general} points in a projective space, then the same phenomenon happens for homogeneous polynomials in several variables, except for a few identified exceptional cases.

Before proceeding, it may be useful to clarify our intention with this survey. There are already a few surveys discussing part of this topic in the literature. For instance, the surveys by C. Ciliberto  \cite{Ci01} and J. Harris \cite{Ha11} introduce (Hermite) interpolation problems and related results, and include brief discussions of some of the geometric ideas behind the Alexander--Hirschowitz theorem. A survey by R. A. Lorentz \cite{Lo} discusses these topics from a more Numerical Analysis  perspective. A survey by M. C. Brambilla and G. Ottaviani \cite{BO} discusses the history and presents many details of the core arguments needed in the proof of the Alexander--Hirschowitz theorem.

These existing surveys assume advanced knowledge and tools from Algebraic Geometry, and are written in languages that may be more familiar to an algebraic geometer (cf. \cite{BO}) or an analyst (cf. \cite{Lo}). Some of the stated facts from these surveys may not appear so obvious for a young reader who is not specifically well trained in algebraic geometry; for instance, the use of curvilinear subschemes and the semi-continuity of Hilbert function. Furthermore, the recent large body of work on symbolic powers of ideals in commutative algebra has drawn our attention and convinced us that it is a good time to reintroduce the Alexander--Hirschowitz theorem to commutative algebraists.

For these reasons, and partly due to a personal interest, our survey is intended for an audience consisting of young commutative algebraists. We aim to present a self-contained proof of the Alexander-Hirschowitz theorem and, particularly, to provide all details that may not be easy to see for commutative algebraists who are new to this research area. We will follow an approach similar to the one of \cite{BO}. However, our style of presentation reflects our choices in using algebraic notions and techniques. At the same time, we still identify and appreciate the fundamental geometric ideas at the core of the proof.

We should mention that, to the best of our knowledge, there is no survey or paper with a completely self-contained proof of the Alexander--Hirschowitz theorem. While \cite{BO} does include many details of the core argument, its emphasis is geared towards techniques that historically have been used to approach the Interpolation Problem, and the tight connections of this problem with secant varieties. We have also discovered in the literature a few computational inaccuracies and incorrect statements; while they are minor, yet a rechecking was required. Additionally, we shall include all necessary tools in a few appendices; there we state basic results on symbolic powers of ideals, secant varieties, Hilbert functions, generic points, curvilinear schemes and the semi-continuity of Hilbert function.

The proof we present in this survey incorporates all up-to-date simplifications of the arguments in the original proof of the Alexander--Hirschowitz theorem, including, for instance, the work done by K. Chandler \cite{Ch00, Ch02}, and Brambilla and Ottaviani \cite{BO} (regarding the case of cubics).

We shall now give a number of important notations and terminology needed to state the Alexander--Hirschowitz theorem. Fix a positive integer $n$ and let $R = \CC[x_0,\ldots,x_n]=\CC[\PP^n]$ be the homogeneous coordinate ring of $\PP^n=\PP^n_\CC$. For a zero-dimensional subscheme $X \subseteq \PP^n$, let $I_X \subseteq R$ denote its defining ideal. It is a basic fact the Hilbert function $H_{R/I_X}$ of $R/I_X$ is bounded above by its \emph{multiplicity} $e(R/I_X)$ and the Hilbert function of $R$ (see Corollary \ref{cor.count2}). Particularly, $H_{R/I_X}(d) \le \min\left\{e(R/I_X), {n+d \choose d}\right\}$ for all $d \in \NN$. We say that a zero-dimensional subscheme $X$ in $\PP^n$ {\em has maximal Hilbert function in degree $d$}, or simply \emph{is $\text{AH}_n(d)$},
if
	$$H_{R/I_X}(d) = \min\left\{e(R/I_X), {n+d \choose d}\right\}.$$
This is equivalent to what is often referred to as \emph{imposing independent conditions on degree $d$ hypersurfaces in $\PP^n$}. This latter name however could be slightly misleading because the given property is equivalent to the linear system of equations associated to the points having maximal rank; it is not equivalent to the stronger property that these equations are linearly independent. Thus, we choose to use the notation $\text{AH}_n(d)$, which has essentially been used already in \cite{Ch00, BO}.

 Another basic fact about Hilbert function of zero-dimensional subschemes in $\PP^n$, see Propositions \ref{e(M)} and \ref{e(R/I)2}, is that
 $H_{R/I_X}(d) = e(R/I_X) \ \text{ for all } d \gg 0.$
Thus, we say that $X$ \emph{is multiplicity $d$-independent} if
$$H_{R/I_X}(d) = e(R/I_X).$$
In the known literature, this property is commonly referred to as being simply \emph{$d$-independent}. We add the word ``multiplicity'' to the terminology to emphasize the fact that the Hilbert function of $R/I_X$ at degree $d$ equals its multiplicity, in this case, and to avoid the potential confusion between the similar-sounding properties of {\em imposing independent conditions in degree $d$} and {\em being $d$-independent}.

Let $Y = \{P_1, \dots, P_r\}$ be a set of distinct points in $\PP^n$ and suppose that the defining ideal of $P_i$ is $\pp_i \subseteq R$ for all $i = 1, \dots, r$. Then, the defining ideal of $Y$ is $I_Y = \pp_1 \cap \dots \cap \pp_r$. A celebrated theorem of Zariski and Nagata (Theorem \ref{ZariskiNagata}) implies that the \emph{symbolic square} $I_Y^{(2)} = \pp_1^2 \cap \dots \cap \pp_r^2$ consists of all homogeneous polynomials in $R$ passing through each point of $Y$ at least twice. Let $X$ be the zero-dimensional subscheme in $\PP^n$ defined by $I_Y^{(2)}$. We call $X$ the set of $r$ \emph{double points supported on $Y$}, and write $X = 2Y = \{2P_1, \dots, 2P_r\}$ for simplicity of notation. 
$X=2Y$ is called a \emph{general} set of $r$ double points if $Y$ is a general set of $r$ simple points (see Definition \ref{DefGeneral} for the precise definition of general sets of simple points).

We are ready to state the main theorem surveyed in this paper.

\begin{Theorem}[Alexander-Hirschowitz]\label{AH}
	\label{AH}
	Let $n,d$ be positive integers. Let $X$ be a general set of $r$ double points in $\PP^n_\CC$. Then, $X$ is $\text{AH}_n(d)$ with the following exceptions:
	\begin{enumerate}
		\item $d=2$ and $2 \le r \le n$;
		\item $d=3$, $n=4$ and $r = 7$; and
		\item $d=4$, $2 \le n \le 4$ and $r = {n+2 \choose 2}-1$.
	\end{enumerate}
\end{Theorem}

Our proof of Theorem \ref{AH} follows an outline similar to the one of \cite{BO}. Theorem \ref{AH} is proved by double-induction, on $n$ and $d$. For sporadic small values of $n$ and $d$ the inductive hypotheses are not satisfied. Some of these sporadic cases are indeed the exceptions appearing in the statement, but the other ones are not and they are checked to be $\text{AH}_n(d)$ on an {\em ad hoc} basis. In general, for the inductive step, two fundamental ingredients of the proof are the so-called \emph{m\'ethode d'Horace diff\'erentielle} and the use of 0-dimensional schemes of prescribed length and with support on a set of points. Their refined and delicate use is at the core of the simplifications of the original proof.

We now outline the structure of this survey. In section \ref{sec.core}, we present the proof of Theorem \ref{AH} for $d \geq 4$ and $n \ge 2$, when induction works. This is the most technical section of the paper. We will summarize the main ideas behind the core inductive argument before giving the details of this inductive step in Theorem \ref{core}.  Theorem \ref{core} is then employed to prove Theorem \ref{AH} as well as other results in the survey.
In section \ref{exceptional}, we discuss the  exceptional cases, leaving out some details when $n = 2$ and when $d = 3$ until later in Sections \ref{sec.P2} and \ref{sec.cubic}. In Section \ref{sec.P2}, we give the proof of Theorem \ref{AH} when $n = 2$, i.e., for points on the projective plane. We have chosen to write a proof which employs Theorem \ref{core} to provide the reader with another illustration of the use of this core inductive argument. In Section \ref{sec.cubic}, we conclude the proof of the Alexander--Hirschowitz theorem by examining the case when $d=3$, i.e., for cubics. The paper continues with a list of open problems and questions in Section \ref{sec.open}.

As mentioned, we end the paper with a number of short appendices to complement the previous sections. In Appendix \ref{app.secant}, we briefly illustrate the connection between the (homogeneous, Hermite) double interpolation problem and computing the dimension of certain secant varieties as well as determining the Waring rank of forms.
In Appendix \ref{app.Symbolic}, we recall the definition of symbolic powers and the statement of a fundamental theorem of Zariski and Nagata drawing the connection between symbolic powers of ideals of points and the interpolation problems. Since this paper is largely about the Hilbert function of zero-dimensional subschemes in $\PP^n$, we have included an appendix about Hilbert functions and, especially, the \emph{lower semi-continuity} property of Hilbert function; see Appendices \ref{app.Hilbertfn} and \ref{app.semi-cont}.
The proof of Theorem \ref{AH} uses a number of known facts about \emph{Hilbert schemes of points} and \emph{curvilinear subschemes}, which may not be obvious for an algebraist (they were not obvious for us), so we include an appendix about Hilbert schemes and curvilinear subschemes; see Appendix \ref{app.curvilinear}.

Finally, for sake of clarity, we have chosen to work over $\CC$, however, a large number of results would still be valid over any perfect field (and using divided powers rather than the usual derivatives, in case the characteristic of the field is positive).

\begin{acknowledgement} The authors would like to thank Irena Peeva for the invitation to write a paper for this volume. The first author is partially supported by Louisiana Board of Regents (grant \#LEQSF(2017-19)-ENH-TR-25).
\end{acknowledgement}


\section{The general case ($d\geq 4$ and $n\geq 3$)} \label{sec.core}

In this section, we discuss the core inductive argument for the proof of Theorem \ref{AH}. It is known, see Proposition \ref{e(R/I)2}, that if $X$ is a set of $r$ double points in $\PP^n$ then
$e(R/I_X) = r(n+1).$ Thus, a set $X$ of $r$ double points in $\PP^n$ is $\text{AH}_n(d)$ if and only if
$$H_{R/I_X}(d) = \min\left\{{n+d\choose n}, r(n+1)\right\}.$$
The following observations allow us to \emph{specialize}, i.e., to deduce the statement of Theorem \ref{AH} by constructing a specific set $X$ of $r$ double points in $\PP^n$ which is $\text{AH}_n(d)$, and to consider at most two values of $r$.

\begin{Remark} \label{rmk.Hsemi}
Fix $n,d\in \ZZ_+$.
\begin{itemize}
\item (Corollary \ref{special}) If there exists \emph{one} collection of $r$ double points in $\PP^n$ that is $\text{AH}_n(d)$, then any general set of $r$ double points in $\PP^n$ is $\text{AH}_n(d)$.	
\item (Corollary \ref{rpts}) To prove that any set of $r$ general double points in $\PP^n$ is $\text{AH}_n(d)$, it suffices to verify the statement for the following two (possibly coinciding) values of $r$:
	$$\left\lfloor \dfrac{1}{n+1}{n+d \choose d}\right\rfloor \le r \le \left\lceil \dfrac{1}{n+1}{n+d \choose d}\right\rceil.$$
\end{itemize}
\end{Remark}

A key ingredient for the inductive argument of Theorem \ref{AH} is the so--called Castelnuovo's Inequality which we now recall. 

\begin{Lemma}[Castelnuovo's Inequality] \label{lem.Cast}
Let $R$ be a polynomial ring. Let $I$ be a homogeneous ideal and let $\ell$ be a linear form in $R$. Set $\widetilde{I}=I:\ell$, $\overline{R} = R/(\ell)$, and $\overline{I} = I \overline{R}$. Then,
\begin{align}
H_{R/I}(d) \geq H_{R/\widetilde{I}}(d-1) + H_{\ovl{R}/(\ovl{I})^{\text{sat}}}(d). \label{eq.Castelnuovo}
\end{align}
Additionally, the equality holds for every $d$ if and only if $\ovl{I}$ is saturated in $\ovl{R}$.
\end{Lemma}

\begin{proof}
From the standard exact sequence
$$
0 \lra R/I:\ell(-1) \stackrel{\cdot \ell}{\lra} R/I \lra R/(I,\ell) \lra 0,
$$
and the fact that $\ovl{I}\subseteq (\ovl{I})^{\text{sat}}$, one obtains
$$
H_{R/I}(d) = H_{R/\widetilde{I}}(d-1) + H_{\ovl{R}/\ovl{I}}(d) \geq H_{R/\widetilde{I}}(d-1) + H_{\ovl{R}/(\ovl{I})^{\text{sat}}}(d).
$$
It is also clear that the equality holds for every $d$ if and only if $H_{\ovl{R}/\ovl{I}}(d) = H_{\ovl{R}/(\ovl{I})^{\text{sat}}}(d)$ for all $d$, which is the case if and only if $\ovl{I} = (\ovl{I})^{\text{sat}}$.
\end{proof}

An intuitive natural approach to Theorem \ref{AH} is to apply Casteluovo's Inequality to obtain a proof by induction on $n\geq 1$. Indeed, Terracini already employed this method to study the case of $n=3$ by partly reducing to the case of $n=2$. We shall capture the modern version of Terracini's argument.

\begin{Theorem}[Terracini's Inductive Argument] \label{thm.Terr}
Fix integers $r\geq q\geq 1$ and $d\in \ZZ_+$ satisfying either
$$r(n+1) - {d+n-1 \choose n} \leq qn \leq {d+n-1\choose n-1}\quad  \text{ or }\quad{d+n-1\choose n-1} \leq qn \leq r(n+1) - {d+n-1 \choose n}.$$
Let $L$ be a hyperplane in $\PP^n$. If \begin{enumerate}
\item a set of $q$ general double points in $L \simeq \PP^{n-1}$ is $\text{AH}_{n-1}(d)$, and
\item the union of a set of $r-q$ general double points in $\PP^n$ and a set of $q$ general simple points in $L$ is $\text{AH}_n(d-1)$,
\end{enumerate}
then a set of $r$ general double points in $\PP^n$ is $\text{AH}_n(d)$.
\end{Theorem}

\begin{proof}
Without loss of generality, we may assume that $x_n=0$ is the equation of $L$. Let $R=\CC[x_0,\ldots,x_n]$, and let $\ovl{R}=\CC[x_0,\ldots,x_{n-1}]\simeq R/(x_n)$. Let $Y_1$ be a set of $q$ general simple points in $L \simeq \PP^{n-1}$, with defining ideal $\ovl{I_{Y_1}}\subseteq \ovl{R}$. If we consider $Y_1$ as a set of points in $\PP^n$, then its defining ideal is $I_{Y_1}:=(\ovl{I_{Y_1}}, x_n)R$. Let  $Y_2$ be a set of $r-q$ general simple points in $\PP^n - L$ with defining ideal $I_{Y_2}\subseteq R$. Let $I=I_{Y_1}^{(2)} \cap I_{Y_2}^{(2)}$ be the defining ideal of $2Y=2Y_1 \cup 2Y_2$, and $\ovl{I}=I\ovl{R}$.
By Remark \ref{rmk.Hsemi} and Corollary \ref{cor.count2}, it suffices to show that $H_{R/I} \ge \min\left\{{n+d\choose n}, r(n+1)\right\}$.

Since $Y_1\subseteq L$ and none of the points in $Y_2$ lies on $L$, we have
$$
\widetilde{I}:=I:x_n = (I_{Y_1}^{(2)}: x_n)  \cap (I_{Y_2}^{(2)} : x_n) = I_{Y_1} \cap I_{Y_2}^{(2)},
$$
which is the defining ideal of the union of a set of  $q$ general simple points in $H$ and $r-q$ general double points in $\PP^n$. Next we show that $(\ovl{I})^{sat} = \ovl{I_{Y_1}}^{(2)}$. First, observe that  $\h(\ovl{I})=\dim(\ovl{R})-1$, so $(\ovl{I})^{\sat}$ is the intersection of the minimal components of $\ovl{I}$. These minimal components are the images in $\ovl{R}$ of the minimal components of $(I,x_n)$. Now, the primes containing $(I,x_n)=(I_{Y_1}^{(2)} \cap I_{Y_2}^{(2)}, x_n)$ are precisely the primes containing $x_n$ and either $I_{Y_1}$ or $I_{Y_2}$. Since for any $\p\in {\rm Min}(I_{Y_2})$ we have $x_n\notin \p$, then $(\p,x_n)=(x_0,\ldots,x_n)$ is the maximal ideal of $R$ and, thus, it is not a minimal prime of $(I,x_n)$ (which has height $n$). On the other hand, for any $\p\in {\rm Min}(I_{Y_1})$ we have $\h(\p)=n$ and $x_n\in \p$, so $\p\in {\rm Min}(I,x_n)$. It follows that the minimal primes $\p$ of $(I,x_n)$ are precisely the minimal primes of $I_{Y_1}$ and when we localize at any of them we get $(I,x_n)_{\p} = (I_{Y_1}^{(2)},x_n)_{\p}$. It follows that $(I,x_n)^{sat} = (I_{Y_1}^{(2)},x_n)^{sat}$ and, by taking images in $\ovl{R}$, we derive that $(\ovl{I})^{sat} = \ovl{I_{Y_1}}^{(2)}$.

Now, by assumptions (1) and (2), we have
$$
 H_{R/\widetilde{I}}(d-1)=\min\left\{{n+d-1\choose n}, q + (n+1)(r-q)\right\}$$
 and
 $$H_{\ovl{R}/\ovl{I}^{\sat}}(d) = \min\left\{{(n-1)+d\choose n-1}, qn\right\}.
$$
These inequalities together with Lemma \ref{lem.Cast} yield
$$
\begin{array}{ll}
H_{R/I}(d) & \geq H_{R/\widetilde{I}}(d-1) + H_{\ovl{R}/\ovl{I}^{\sat}}(d)\\
		& = \min\left\{{n+d-1\choose n}, q + (n+1)(r-q)\right\} + \min\left\{{n-1+d\choose n-1}, qn\right\}.
\end{array}
$$

Now, if $r(n+1) - {d+n-1 \choose n} \leq qn \leq {d+n-1\choose n-1}$, then $ \min\left\{{n+d\choose n}, r(n+1)\right\}=r(n+1)$, and
$$
\min\left\{{n+d-1\choose n}, q + (n+1)(r-q)\right\} + \min\left\{{n-1+d\choose n-1}, qn\right\} =q + (n+1)(r-q) +  qn = r(n+1).
$$
Thus,
$$
H_{R/I}(d) \geq r(n+1) = \min\left\{{n+d\choose n}, r(n+1)\right\}.
$$
Similarly, if ${d+n-1\choose n-1} \leq qn \leq r(n+1) - {n-1+d \choose n}$ holds, then
$$
\begin{array}{ll}
H_{R/I}(d) & \geq\min\left\{{n+d-1\choose n}, q + (n+1)(r-q)\right\} + \min\left\{{n-1+d\choose n-1}, qn\right\}\\	
		& = {n+d-1\choose n} + {n-1+d\choose n-1}\\
		& = {n+d\choose n}\\
		& = \min\left\{{n+d\choose n}, r(n+1)\right\}.
\end{array}
$$
This concludes the proof.
\end{proof}

Assumption (1) in Theorem \ref{thm.Terr} is usually provided by the inductive hypothesis.
Assumption (2) is more delicate, because we have a mix of double points and simple points --- Proposition \ref{prop.hyperp} provides the tool to handle this situation.
What prevents one from using Theorem \ref{thm.Terr} to prove Theorem \ref{AH} is the fact that there may not be an integer $q$ satisfying both of the numerical assumptions of Theorem \ref{thm.Terr}.
For instance, to prove the case where $n=3$ and $d=6$, by Remark \ref{rmk.Hsemi}, we need to prove that a set of $r=21$ general double points satisfies $\text{AH}_3(6)$. To apply Theorem \ref{thm.Terr}, we need to find $q\in \ZZ$ with $84-56 \leq 3q \leq 28$, i.e. $q=28/3$. Thus, Theorem \ref{thm.Terr} is not applicable. There are in fact infinitely many choices of $n$ and $d$ for which we run into the same problem, i.e., when we cannot apply Theorem \ref{thm.Terr} directly.

The \emph{m\'ethode d'Horace diff\'erentielle} of \cite{AH1} is designed to overcome this difficulty. For a subscheme $X \subseteq \PP^n$ and a hyperplane $L$ defined by a linear form $\ell$, we use $\widetilde{X}$ to denote the \emph{residue} of $X$ with respect to $L$; that is, the subscheme of $\PP^n$ defined by the ideal $I_X : \ell$. The underlying ideas of the {\em m\'ethode d'Horace diff\'erentielle} are:
\begin{enumerate}
	\item[Step 1.] Fix a hyperplane $L \simeq \PP^{n-1}$ in $\PP^n$. For a suitable choice of $q$ and $\epsilon$, choose a general collection $2\Psi$ of $r-q-\epsilon$ double points not in $L$, a general collection $2\Lambda$ of $q$ double points in $L$, and a general collection $2\Gamma$ of $\epsilon$ double points in $L$.
	\item[Step 2.] By induction on the dimension, the sets $2\Lambda \cup 2\Gamma_{|L}$ and $\Psi \cup 2\Lambda \cup 2\Gamma_{|L}$ have maximal Hilbert function in degree $(d-1)$ in $L \simeq \PP^{n-1}$. One shows that to prove the theorem it suffices to prove that $2\Gamma$ is multiplicity $[I_{2\Psi \cup 2\Lambda}]_d$-independent (see Definition \ref{def:mult} below).
	\item[Step 3.] The last statement in Step 2 is proved using deformation. For $\t = (t_1, \dots, t_\epsilon) \in K^\epsilon$ we take a flat family of general points $\Gamma_\t$ lying on a family of hyperplanes $\{L_{t_1}, \dots, L_{t_\epsilon}\}$ having $\Gamma$ as a limit when $\t\lra 0$, and the problem reduces to showing that  $2\Gamma_{\t}$ is multiplicity $[I_{2\Psi \cup 2\Lambda}]_d$-independent for some $\t$.
	\item[Step 4.] To establish this latter fact, the existence of $\t$ in Step 3, we argue by contradiction and another deformation argument reduces the problem to understanding the Hilbert function of schemes of the form $2\Psi \cup 2\Lambda \cup \Theta_\t$, for a suitable curvilinear subscheme $\Theta_\t$ supported on $\Gamma_\t$ and contained in $2\Gamma_\t$ (see Appendix \ref{app.curvilinear} for basic facts about curvilinear schemes). Since $\Gamma_\t$ is a family of curvilinear schemes, the family has a limit which can be used in the process. Finally, arguments employing the semi-continuity of the Hilbert function, the Castelnuovo inequality (\ref{eq.Castelnuovo}) and the material developed in Step 2 allows us to arrive at the desired conclusion.
\end{enumerate}

The deformation argument in Step 4 of the {\em m\'ethode d'Horace diff\'erentiell}e is possible by the use of curvilinear subschemes and, particularly, Lemma \ref{lem.reductionCL}, which we shall now introduce.

\begin{Definition}
	Let $V$ be a $\CC$-vector space of homogeneous polynomials of the same degree in $R = \CC[x_0, \dots, x_n]$ and let $I \subseteq R$ be a homogeneous ideal. Let $I \cap V$ denote the $\CC$-vector space of forms (necessarily of the same degree) belonging to both $I$ and $V$.
\end{Definition}

Recall that  a zero-dimensional subscheme  $X \subseteq \PP^n$ is \emph{multiplicity $d$-independent} if
$H_{R/I_X}(d) = e(R/I_X).$ 
\begin{Definition}\label{def:mult}
	Let $X \subseteq \PP^n$ be a zero-dimensional subscheme and let $V$ a $\CC$-vector space of homogeneous polynomials of the same degree in $R$.
	\begin{enumerate}
		\item The \emph{Hilbert function of $X$ (or $I_X$) with respect to $V$} is defined to be
	$$h_{\PP^n}(X,V) = \dim_\CC V- \dim_\CC(I_X \cap V).$$
	\item We say that $X$ (or $I_X$) is \emph{multiplicity $V$-independent} if
	$$h_{\PP^n}(X,V) = e(R/I_X).$$
	\end{enumerate}
\end{Definition}
This definition generalizes multiplicity $d$-independence in the sense that $X$ is multiplicity $R_d$-independent if and only if $X$ is multiplicity $d$-independent. We now prove a couple of basic facts.

\begin{Lemma}\label{multind}
		Let $X \subseteq \PP^n$ be a zero-dimensional subscheme and let $V$ a $\CC$-vector space of homogeneous polynomials of the same degree in $R$. Then\begin{enumerate}
			\item  $h_{\PP^n}(X,V) \le \min\{H_{R/I_X}(d), \dim_\CC V\};$
			\item if $X$ is multiplicity $V$-independent then $X$ is multiplicity $d$-independent; 
			\item if $X$ is multiplicity $d$-independent then so is $Y$, for any zero-dimensional subscheme $Y$ of $X$.
		\end{enumerate}
\end{Lemma}

\begin{proof}
	(1)  Let $[I_X]_d := I_X\cap R_d$, thus we have $(I_X\cap V)=[I_X]_d\cap V$ and $H_{I_X}(d) = \dim_\CC [I_X]_d$. By definition $h_{\PP^n}(X,V) = \dim_\CC V- \dim_\CC(I_X \cap V)\leq \dim_{\CC} V$, so we only need to prove $h_{\PP^n}(X,V)\leq H_{R/I_X}(d)$. From the short exact sequence of vector spaces
	$$
	0 \lra I_X \cap V \lra [I_X]_d \oplus V \lra [I_X]_d + V \lra 0,
	$$
	and the additivity of dimension of vector spaces we obtain that $h_{\PP^n}(X,V) = \dim_\CC ([I_X]_d + V) - H_{I_X}(d)$, which is at most $H_{R/I_X}(d)$ because $[I_X]_d + V\subseteq R_d$. 
	
	(2)  By (1) and the fact that $H_{R/I_X}(d) \leq e(R/I_X)$ for every $d$ (see Proposition \ref{e(M)}) we have $h_{\PP^n}(X,V)\leq H_{R/I_X}(d) \leq e(R/I_X)$. So if $X$ is multiplicity $V$-independent then $h_{\PP^n}(X,V) = e(R/I_X) = H_{R/I_X}(d)$ and, particularly, $X$ is also multiplicity $d$-independent.
	
	(3) Since $R/I_X$ and $R/I_Y$ are 1-dimensional Cohen-Macaulay modules, the short exact sequence 
	$$
	0 \lra I_Y/ I_X \lra R/I_X \lra R/I_Y \lra 0
	$$
	implies that $I_Y/I_X$ is a 1-dimensional Cohen-Macaulay module and $e(I_Y/I_X)=e(R/I_X) - e(R/I_Y)$. 
	Now, by the above short exact sequence and Proposition \ref{e(M)}, we obtain
	$$ H_{R/I_Y}(d)  = e(R/I_X) - H_{I_Y/I_X}(d) \geq e(R/I_X) - [e(R/I_X) - e(R/I_Y)]=e(R/I_Y).$$
	Since $H_{R/I_Y}(d)\leq e(R/I_Y)$, by Proposition \ref{e(M)}, we conclude that $H_{R/I_Y}(d) = e(R/I_Y)$.
\end{proof}

Let $Z$ be a set of finitely many simple points, we shall now prove that to check whether a scheme $X$ contained in $2Z$ is $V$-independent it suffices to consider curvilinear subschemes of $X$. This reduction and the fact that curvilinear schemes form a dense open subset of the Hilbert scheme (see Proposition \ref{prop:dense}) play an important role in the proof of Theorem \ref{core}.

\begin{Lemma}[Curvilinear Lemma] \label{lem.reductionCL}
	Let $X \subseteq \PP^n$ be a zero-dimensional scheme contained in a finite union of double points and let $V$ be a $\CC$-vector space of homogeneous polynomials of degree $d$ in $R$. Then $X$ is multiplicity $V$-independent if and only if every curvilinear subscheme of $X$ is multiplicity $V$-independent.
\end{Lemma}

\begin{proof} Suppose first that $X$ is multiplicity $V$-independent, then $X$ is multiplicity $d$-independent by Lemma \ref{multind}(2)  and 
$h_{\PP^n}(X,V) = e(R/I_X) = H_{R/I_X}(d).$
Particularly, $V$ contains all homogeneous polynomials of degree $d$ that are not in $I_X$. Let $Y\subseteq X$ be any 0-dimensional subscheme, then clearly $V$ contains all homogeneous polynomials of degree $d$ that are not in $I_Y$, and $Y$ is multiplicity $d$-independent by Lemma \ref{multind}(3).
Therefore, 
$$h_{\PP^n}(X,V) - h_{\PP^n}(Y,V) = \dim_\CC (I_Y \cap V) - \dim_\CC (I_X \cap V) = H_{R/I_X}(d) - H_{R/I_Y}(d) = e(R/I_X) - e(R/I_Y).$$
Since by assumption $h_{\PP^n}(X,V)=e(R/I_X)$, then $h_{\PP^n}(Y,V) = e(R/I_Y)$, and so $Y$ is multiplicity $V$-independent.

Suppose now that every curvilinear subscheme of $X$ is multiplicity $V$-independent. We shall
use induction on the number $r$ of points in the support of $X$ and $e(R/I_X)$ to show that $h_{\PP^n}(X,V) = e(R/I_X)$.
	
\textsc{Case 1:} $X$ is supported at a single point $P \in \PP^n$. If $e(R/I_X) = 1$ then $X = \{P\}$ and the statement is trivial. If $e(R/I_X) = 2$ then, locally at $P$, $X \cong \Spec(T)$ where $T$ is a local $\CC$-algebra of vector space dimension 2 over $\CC$. This implies that the maximal ideal $\m$ of $T$ is of vector space dimension 1 over $\CC$ and $\m^2 = 0$. It follows that $T \cong \CC[t]/(t^2)$. As a consequence (see Lemma \ref{lem.curvilinear}), $X$ is a curvilinear scheme. Therefore, $X$ is multiplicity $V$-independent by the hypotheses.

Assume that $e(R/I_X) > 2$. Let $Y \subseteq X$ be any subscheme with $e(R/I_Y) = e(R/I_X) - 1$. Clearly, $h_{\PP^n}(Y,V) \le h_{\PP^n}(X,V)$. Observe that any curvilinear subscheme of $X$ restricts to a curvilinear subscheme of $Y$. Thus, by the induction hypothesis, we conclude that $Y$ is multiplicity $V$-independent. That is,
$$h_{\PP^n}(Y, V) = e(R/I_Y) = e(R/I_X) - 1.$$
Particularly, this implies that $h_{\PP^n}(X,V) \le e(R/I_X) = h_{\PP^n}(Y,V)+1$. Thus, to show that $X$ is multiplicity $V$-independent it suffices to construct a subscheme $Y$ of $X$ such that $e(R/I_Y) = e(R/I_X) - 1$ and $h_{\PP^n}(X,V) = h_{\PP^n}(Y,V) + 1$ (equivalently, $h_{\PP^n}(X,V) > h_{\PP^n}(Y,V)$).

To this end, let $\zeta \subseteq X$ be a subscheme of multiplicity 2. As shown above $\zeta$ is a curvilinear subscheme of $X$. Thus, $\zeta$ is multiplicity $V$-independent, i.e., $h_{\PP^n}(\zeta, V) = 2$. On the other hand, $h_{\PP^n}(P, V) \le H_{R/I_P}(d) = 1$, by Lemma \ref{multind}(1). Therefore, there exists a homogeneous polynomial $f$ in $V$ that vanishes at $P$ but not on $\zeta$. Set $Z = \mathbb{V}(f)$ be the zero locus of $f$, and define $Y = X \cap Z$. Since $X$ is contained in $2P$, by imposing the condition that $f = 0$ on $Y$, we have $e(R/I_Y) = e(R/I_X) - 1$. Furthermore, $f$ vanishes on $Y$ but not on $X$, and so $h_{\PP^n}(X,V) > h_{\PP^n}(Y,V)$. The assertion follows in this case.

\textsc{Case 2:} $X$ is supported at $r$ points $P_1, \dots, P_r$ for $r \ge 2$. By induction on $r$, we may assume that the statement is true for schemes supported at $r-1$ points. 

 Let $I:=I_X= \q_1 \cap \ldots  \cap \q_r$ be an irredundant primary decomposition of $I=I_X$ and let $\p_i=\sqrt{\q_i}$ for every $i$. Let $\q:=\q_r$, let $Q$ be its associated scheme and let $Z\subseteq X$ be the scheme defined by $I_Z:=\q_1\cap \ldots \cap \q_{r-1}$. By assumption, every curvilinear scheme contained in $X$ is $V$-independent, and then so is every curvilinear scheme contained in $Z$. Since $Z$ is supported at $r-1$ points, by inductive hypothesis we have 
$$h_{\PP^n}(Z, V) = e(R/I_Z).$$

\noindent {\bf Claim 1.} $X$ is $V$-independent if one proves that $Q$ is $V\cap [I_Z]_d$-independent.\\
\\
\textit{Proof of Claim 1.} To prove that $X$ is $V$-independent we compute
$$
\begin{array}{ll}
h_{\PP^n}(X,V) & = \dim_{\CC}V - H_{I_Z \cap \q\cap V}(d)\\
		&  = [\dim_{\CC}V - H_{I_Z\cap V}(d)] + [H_{I_Z\cap V}(d) - H_{I_Z \cap \q\cap V}(d)]\\		& = h_{\PP^n}(Z,V)  + [H_{I_Z\cap V}(d) - H_{I_Z \cap \q\cap V}(d)]\\
		& = e(R/I_Z) +  [H_{I_Z\cap V}(d) - H_{I_Z\cap q\cap V}(d)].\\
\end{array}
$$
If $Q$ is $V\cap [I_Z]_d$-independent, then $H_{I_Z\cap V}(d) - H_{I_Z \cap \q\cap V}(d)=e(R/\q)$, and thus $h_{\PP^n}(X,V) = e(R/I_Z) + e(R/\q)=e(R/I)$.\\
\\ 
{\bf Claim 2.} It suffices to prove that $Z\cup Q'$ is $V$-independent for any curvilinear scheme $Q'\subseteq Q$.\\
\\
\textit{Proof of Claim 2.} Let $\q'\supseteq \q$ be the defining ideal of $Q'\subseteq Q$. By Claim 1 it suffices to prove that $Q$ is $V\cap [I_Z]_d$-independent. By the base case of induction, it suffices to prove that $Q'$ is $V\cap[I_Z]_d$-independent for any curvilinear scheme $Q'\subseteq Q$. We compute $h_{\PP^n}(Q', V\cap [I_Z]_d)$:
$$
\begin{array}{ll}
h_{\PP^n}(Q', V\cap [I_Z]_d)&= [\dim_{\CC}V - H_{I_Z\cap \q' \cap V}(d)] - [\dim_{\CC}V - H_{I_Z\cap V}(d) ]\\
& = h_{\PP^n}(Z\cup Q', V) - h_{\PP^n}(Z,V) \\
& = h_{\PP^n}(Z\cup Q', V) - e(R/I_Z).
\end{array}
$$
Therefore, if $Z\cup Q'$ is $V$-independent, then $h_{\PP^n}(Z\cup Q', V) = e(R/I_Z\cap \q')$ and, by the above computation, $h_{\PP^n}(Q', V\cap [I_Z]_d)= e(R/I_Z\cap \q') - e(R/I_Z)=e(R/\q')$, proving that $Q'$ is $V\cap [I_Z]_d$-independent. This establishes Claim 2.\\
\\
{\bf Claim 3.} It suffices to show that for any curvilinear scheme $Z' \subseteq Z$ one has $Z'$ is $V\cap [\q']_d$-independent.\\
\\
\textit{Proof of Claim 3.} Recall that by Claim 2 it suffices to prove $Z\cup Q'$ is $V$-independent. We observe that
$$
\begin{array}{ll}
h_{\PP^n}(Z\cup Q', V)  & = \dim_{\CC}V - H_{I_Z\cap \q' \cap V}(d)\\
&  = [\dim_{\CC}V - H_{\q'\cap V}(d)] + [H_{\q'\cap V}(d) - H_{I_Z\cap \q' \cap V}(d)] \\
& = h_{\PP^n}(Q', V) +h_{\PP^n}(Z, V\cap [Q']_d).
\end{array}
$$
Since $Q'\subseteq Q \subseteq X$ is curvilinear, $Q'$ is $V$-independent by the assumption, i.e., $h_{\PP^n}(Q',V) = e(R/\q')$. Thus, it suffices to prove that $h_{\PP^n}(Z, V\cap [Q']_d) = e(R/I_Z)$, because then, by the above computation, $h_{\PP^n}(Z\cup Q', X) = e(R/\q') + e(R/I_Z) = e(R/I_Z\cap \q')$, exhibiting that $Z\cup Q'$ is $V$-independent.

Since $Z$ is supported at $r-1$ points, by the inductive hypothesis, to prove that $Z$ is $V\cap [Q']_d$-independent it suffices to show that $Z'$ is $V\cap [Q']_d$-independent for any curvilinear subscheme $Z'\supseteq Z$. This proves Claim 3.\\
\\
We conclude the proof of Lemma \ref{lem.reductionCL} by showing that $Z'$ is $V\cap [\q']_d$-independent. First, we compute  $h_{\PP^n}(Z, V\cap [\q']_d)$:
$$
\begin{array}{ll}
h_{\PP^n}(Z, V\cap [\q']_d) &  = H_{\q'\cap V}(d) - H_{I_{Z'} \cap \q'\cap V}(d) \\
			&  = [\dim_{\CC}V - H_{I_{Z'}\cap \q'\cap V}(d)] - [\dim_{\CC}V - H_{\q'\cap V}(d)] \\
			& = h_{\PP^n}(Z'\cup Q', V) - h_{\PP^n}(Q', V).\\
\end{array}
$$
Since $Q'\subseteq Q\subseteq X$ is curvilinear, $h_{\PP^n}(Q', V) = e(R/\q')$ by the assumption. Since $Q'$ and $Z'$ are curvilinear and have disjoint support (because $\Ass(R/\q')=\{\p_r\}$ and $\Ass(R/I_{Z'})\subseteq \Ass(R/I_Z)=\{\p_1,\ldots,p_{r-1}\}$), we have that $Z'\cup Q'$ is locally curvilinear at points of the support, and so it is curvilinear. Since $Z'\cup Q'\subseteq Z\cup Q=X$,  $h_{\PP^n}(Z'\cup Q', V)=e(R/I_{Z'}\cap \q')$ by the assumption. Therefore,
$$
h_{\PP^n}(Z', V\cap [\q']_d) = e(R/I_{Z'} \cap \q') - e(R/\q')=e(R/I_{Z'}), 
$$
and Lemma \ref{lem.reductionCL} is established.
\end{proof}

Recall that, by Remark \ref{rmk.Hsemi}, to prove Theorem \ref{AH} for values of $n,d$ not in the list of exceptional cases, it suffices prove that a general set of $r$ double points has $\text{AH}_n(d)$ for
$$
 \Big\lfloor \dfrac{1}{n+1} {n+d \choose n} \Big\rfloor \leq r \leq  \Big\lceil \dfrac{1}{n+1} {n+d \choose n} \Big\rceil.
$$
When $r$ takes one of these two (sometimes coinciding) values, we let $q$ and $\epsilon$ be the quotient and remainder of the division of $r(n+1) - {n+d-1 \choose n}$ by $n$. For ease of references, we now provide the values of $q$ and $\epsilon$ for a few special choices of $n$ and $d$. 

We start with the case where $d = 4$.
\medskip
\begin{center}
\begin{tabular}{|c|c|c|c|c|c|}
	\hline
$n$ & value of $r$ & $\Delta:=r(n+1)- {n+d-1\choose n}$ & value of $q$ & value of $\epsilon$ & value of $r-q-\epsilon$ \\\hline
$n=2$ & $r= 5$  & $\Delta=5$ &  $q=2$ & $\epsilon =1$ & $r-q-\epsilon = 2$\\\hline
$n=3$ & $r=8$  & $\Delta = 12$ & $q=4$ & $\epsilon =0$ & $r-q-\epsilon = 4$\\
           & $r=9$ & $\Delta = 16$ & $q=5$ & $\epsilon =1$ & $r-q-\epsilon = 3$\\\hline
$n=4$ & $r=14$  & $\Delta = 35$ 
 & $q=8$ & $\epsilon =3 $ & $r-q-\epsilon = 3$\\\hline
$n=5$ & $r=21$  & $\Delta = 70$ & $q=14$ & $\epsilon =0 $ & $r-q-\epsilon = 7$\\\hline
$n=6$ &  $r=30$  & $\Delta = 126$ & $q=21$ & $\epsilon =0$ & $r-q-\epsilon = 9$\\\hline
$n=7$ &  $r=41$  & $\Delta = 208$ & $q=29$ & $\epsilon =5$ & $r-q-\epsilon = 7$\\
          &  $r=42$ & $\Delta = 216$ & $q=30$ & $\epsilon =6$ & $r-q-\epsilon = 6$\\\hline
$n=8$ & $r=55$  & $\Delta = 330$ 
 & $q=41$ & $\epsilon =2 $ & $r-q-\epsilon = 12$\\\hline
$n=9$  &  $r=71$  & $\Delta = 490$ & $q=54$ & $\epsilon =4$ & $r-q-\epsilon = 13$\\
          & $r=72$ & $\Delta = 500$ & $q=55$ & $\epsilon =5$ & $r-q-\epsilon = 12$\\\hline
\end{tabular}
\end{center}
\medskip

For $d=5$ and we get the following table.
\begin{center}
\begin{tabular}{|c|c|c|c|c|c|}
	\hline
$n$ & value of $r$ & $\Delta:=r(n+1)- {n+d-1\choose n}$ & value of $q$ & value of $\epsilon$ & value of $r-q-\epsilon$ \\\hline
$n=2$ & $r= 7$  & $\Delta=6$ &  $q=3$ & $\epsilon =0$ & $r-q-\epsilon = 4$\\\hline
$n=3$ & $r=14$  & $\Delta = 21$ & $q=7$ & $\epsilon =0$ & $r-q-\epsilon = 7$\\\hline
$n=4$ & $r=25$  & $\Delta = 55$  & $q=13$ & $\epsilon =3 $ & $r-q-\epsilon = 9$\\
	  & $r=26$  &	$\Delta = 60$  & $q=15$ & $\epsilon =0 $ & $r-q-\epsilon = 11$\\\hline
$n=5$ & $r=42$  & $\Delta = 126$ & $q=25$ & $\epsilon =1 $ & $r-q-\epsilon = 16$\\\hline
$n=6$ &  $r=66$  & $\Delta = 252$ & $q=42$ & $\epsilon =0$ & $r-q-\epsilon = 24$\\\hline
$n=7$ &  $r=99$  & $\Delta = 462$ & $q=66$ & $\epsilon =0$ & $r-q-\epsilon = 33$\\\hline
\end{tabular}
\end{center}
\medskip

We now prove a few basic numeric facts that will be employed later.
\begin{Lemma} \label{lem.numeric}
	For fixed integers $n \ge 2, d \ge 4$ and $0 \le r \le \left\lceil \frac{1}{n+1} {n+d \choose n} \right\rceil$, let $q \in \ZZ$ and $0 \le \epsilon < n$ be such that $nq+\epsilon = r(n+1) - {n+d-1 \choose n}$. Then,
	\begin{enumerate}
		\item $n\epsilon + q \le {n+d-2 \choose n-1}$,
		\item ${n+d-2 \choose n} \le (r-q-\epsilon)(n+1)$,
		\item $r-q-\epsilon \ge n+1$, for $d=4$ and $n\geq 8$. 
		\item $q\geq \epsilon$.
	\end{enumerate}
\end{Lemma}

\begin{proof}
(1) We prove the equivalent statement that $n(n\epsilon + q) \leq n {n+d-2 \choose n-1}$. Clearly, $nq\leq r(n+1) - {n+d-1 \choose n}$. Since $r \le \left\lceil \frac{1}{n+1} {n+d \choose n} \right\rceil$, we have $(n+1)r\leq {n+d \choose n} + n$, and so
\begin{equation}\label{numer1}
n^2 \epsilon + nq \leq n^2(n-1) + {n+d \choose n} + n - {n+d-1 \choose n} = n^2(n-1) + {n+d-1 \choose n-1} + n.
\end{equation}

The right-hand side is at most $n{n+d-2 \choose n-1}$ except when $d=4$ and $3\leq n\leq 5$. In these three cases however the inequality still holds, as one can check directly with given values of $q$ and $\epsilon$ in the above tables.

(2) Since $(r-q-\epsilon)(n+1)  = r(n+1) - (nq+\epsilon) - (n\epsilon + q)$, we have
$$
(r-q-\epsilon)(n+1) = {n+d-1\choose n} - (n\epsilon + q) \geq {n+d-1\choose n} - {n+d-2\choose n-1} = {n+d-2 \choose n},
$$
where the middle inequality follows from (1).

(3) We prove the equivalent statement that $(r-q-\epsilon)(n+1) \ge (n+1)^2$ for $d=4$ and $n\geq 8$. By the computation in (2), $(r-q-\epsilon)(n+1) \geq (n+1)^2$ holds if and only if $\binom{n+3}{n} - (n\epsilon + q) \geq (n+1)^2$. This holds if and only if
$$
(n+1) \left( \frac{(n+3)(n+2)}{6} - (n+1) \right) \geq n\epsilon + q \Llra {n+1\choose 3} \geq n\epsilon + q.
$$
By equation (\ref{numer1}), $n\epsilon + q\leq \frac{1}{n}\left(n^2(n-1) + {n+3 \choose 4} + n\right)$. The right-hand side is at most ${n+1 \choose 3}$ if and only if $n^3-10n^2+3n-10\geq 0$. This inequality holds for all $n\geq 10$. For the cases $n=8,9$ the inequality is easily checked using the above table.

(4) Assume by contradiction that $q<\epsilon$. Then, $r(n+1)-{n+d-1 \choose n} = nq+\epsilon < (n+1)\epsilon\leq (n+1)(n-1)$.
From the definition of $r$, one also sees that $r(n+1)> {n+d \choose n} - (n+1)$, and so $r(n+1) \geq {n+d\choose n}-n$. As a consequence, we have
${n+d\choose n}-n \leq r(n+1) < (n+1)(n-1) + {n+d-1 \choose n}$. Thus, 
${n+d-1 \choose n-1} < (n+1)(n-1) + n$. Since $(n+1)(n-1)+n-1 = (n+2)(n-1)$, this leads to
$
{n+d-1 \choose n-1} \leq (n+2)(n-1).
$

It is well-known that ${n+d-1 \choose n-1}$ increases as $d$ increases, so the left-hand side is at least  ${n+3 \choose n-1}={n+3 \choose 4}$. In particular,
$\frac{(n+3)(n+2)(n+1)n}{24}\leq (n+2)(n-1)$. Therefore, $f(n)\leq 0$, where  $f(n):=(n+3)(n+1)n -24(n-1).$

On the other hand, it is easily seen that $f(n)$ is increasing for $n\geq 2$ and $f(2)=6>0$, thus $f(n)>0$ for every $n\geq 2$, yielding a contradiction.
\end{proof}

We are ready to present the core inductive argument for Theorem \ref{AH}. The proof follows the four steps we outlined when we illustrated the {\em m\'ethode d'Horace diff\'erentielle}.

\begin{Theorem}\label{core}
	For fixed $n \ge 2, d \ge 4$ and ${\displaystyle \Big\lfloor \dfrac{1}{n+1}\binom{n+d}{n}\Big\rfloor \le r \le \Big\lceil \dfrac{1}{n+1} {n+d \choose n} \Big\rceil}$, let $q \in \ZZ$ and $0 \le \epsilon < n$ be such that $nq+\epsilon = r(n+1) - {n+d-1 \choose n}$. Suppose that
	\begin{enumerate}
		\item[(i)] $q$ general double points are $\text{AH}_{n-1}(d)$,
		\item[(ii)] $r-q$ general double points  are $\text{AH}_n(d-1)$,
		\item[(iii)] $r-q-\epsilon$ general double points  are $\text{AH}_n(d-2)$.
	\end{enumerate}
	Then, $r$ general double points are $\text{AH}_{n}(d)$.
\end{Theorem}

\begin{proof} By Remark \ref{rmk.Hsemi}, it suffices to construct a set of $r$ double points in $\PP^n$ which is $\text{AH}_n(d)$. This set of $r$ double points arises in the form $2\Psi \cup 2\Lambda \cup 2\Gamma_\t$, for some family of parameters $\t$, where the sets $\Psi, \Lambda$ and $\Gamma_\t$ are constructed as in the outlined steps. To understand the construction better, we shall use 21 double points in $\PP^3$ and degree 6 as our running example; in this particular situation, $r=21$, $d=6$, $q=9$ and $\epsilon = 1$.
	
\noindent{\bf Step 1.} We first fix a hyperplane $L \simeq \PP^{n-1}$ in $\PP^n$, with defining equation $\ell=0$. We  take a set of $q+\epsilon$ general points in $L$, let $\Gamma = \{\gamma_1, \dots, \gamma_\epsilon\}$ be a subset of $\epsilon$ of these points, and let $\Lambda$ be the set consisting of the remaining $q$ points. Finally, we take a set $\Psi$ of $r-q-\epsilon$ general points in $\PP^n$ outside of $L$. 
(In our running example, $\Gamma$ consists of a single point in $L$, $\Lambda$ of 9 general points in $L$, and $\Psi$ of 11 general points outside of $L$.)

\noindent{\bf Step 2.}	 By (ii), we have
$$
H_{R/(I_{\Psi}^{(2)} \cap I_{\Gamma}^{(2)})}(d-1) = \min \left\{(n+1)(r-q), {n+d-1 \choose n}\right\} =  (n+1)(r-q),
$$
where the rightmost equality holds because Lemma \ref{lem.numeric}(4) yields ${n+d-1 \choose n} = (n+1)(r-q)-\epsilon+q\geq (n+1)(r-q)$.
Now, if we consider $\Gamma|_L$ instead of $\Gamma$, then the linear system associated to $[I_{\Psi}^{(2)} \cap I_{\Gamma|_L}^{(2)}]_{d-1}$ 
is obtained by removing $\epsilon$ equations from the linear system of equations defined by $[I_{\Psi}^{(2)} \cap I_{\Gamma}^{(2)}]_{d-1}$ (more precisely, the ones corresponding to setting the partial derivatives with respect to $\ell$ equal to $0$).
One then obtains
$$
H_{R/(I_{\Psi}^{(2)} \cap I_{\Gamma|_L}^{(2)})}(d-1) = \min \left\{(n+1)(r-q)-\epsilon, {n+d-1 \choose n}\right\}=(n+1)(r-q)-\epsilon
$$
(= 47 for the running example),
and then $H_{I_{\Psi}^{(2)} \cap I_{\Gamma|_L}^{(2)}}(d-1) = {n+d-1 \choose n} - H_{R/(I_{\Psi}^{(2)} \cap I_{\Gamma|_L}^{(2)})}(d-1) = q$ ($=9$ in the running example). \\
\\
{\bf Claim 1.}
$H_{R/(I_{\Psi}^{(2)} \cap I_{\Gamma|_L}^{(2)} \cap I_{\Lambda})}(d-1) = e(R/(I_{\Psi}^{(2)} \cap I_{\Gamma|_L}^{(2)} \cap I_{\Lambda})) = {n+d-1 \choose d-1}$. \\
(that is, $H_{R/(I_{\Psi}^{(2)} \cap I_{\Gamma|_L}^{(2)} \cap I_{\Lambda})}(5) = \binom{3+6-1}{3}=56$  for the running example.)\\
\\
Notice that Claim 1 implies that $2\Psi \cup 2\Gamma_{|L} \cup \Lambda$ is multiplicity $(d-1)$-independent.
The rightmost equality in the claim holds because ${n+d-1 \choose d-1} = {n+d-1 \choose d-1} - \epsilon + q=e(R/(I_{\Psi}^{(2)} \cap I_{\Gamma|_L}^{(2)} \cap I_{\Lambda}))$. To conclude the proof it then suffices to show that $H_{R/(I_{\Psi}^{(2)} \cap I_{\Gamma|_L}^{(2)} \cap I_{\Lambda})}(d-1)={n+d-1 \choose d-1}=H_R(d-1)$, i.e. $[I_{\Psi}^{(2)} \cap I_{\Gamma|_L}^{(2)} \cap I_{\Lambda}]_{d-1}=0$.

If we restate the paragraph before the Claim in terms of linear algebra, we see that the solution set of the linear system defined by $[2\Psi \cup 2\Gamma_{|L}]_{d-1}$ is a $q$-dimensional vector space. Now, for each simple general point in $\PP^n$ that we are adding  to $2\Psi \cup 2\Gamma_{|L}$, we are adding a general linear equation to this system, so we are reducing the dimension of the solution set by 1.
Thus, if we add $q$ general simple points in $\PP^n$ to $2\Psi \cup 2\Gamma_{|L}$, then the corresponding ideal contains no forms of degree $d-1$. It follows that if we add $q$ points to $2\Psi \cup 2\Gamma_{|L}$, and these $q$ additional points lie on $L$, then the defining equation $\ell$ of $L$ divides the equation of any hypersurface of degree $d-1$ passing through $2\Psi \cup 2\Gamma_{|L}$ and these $q$ points. In particular, any form $F \in [I_{\Psi}^{(2)} \cap I_{\Gamma|_L}^{(2)} \cap I_{\Lambda}]_{d-1}$ is divisible by $\ell$, and we can write $F=F_1\ell$.

Since $\ell \in I_{\Lambda}$ (because $\Lambda \subseteq L$) and $\ell$ is regular on $R/I_{\Psi}$ (because none of the points of $\Psi$ lies on $L$), we have that $F_1$ is a degree $(d-2)$ form in
$$
(I_{\Psi}^{(2)} \cap I_{\Gamma|_L}^{(2)} \cap I_{\Lambda}):\ell = I_{\Psi}^{(2)} \cap (I_{\Gamma|_L}^{(2)} :\ell)\subseteq I_{\Psi}^{(2)}.
$$
However, by (iii), we know that $H_{I_{\Psi}^{(2)}}(d-2) = \max \left\{0, {n+d-2 \choose n} - (r-q-\epsilon)(n+1)\right\} = 0$. Therefore, $F_1=0$, and so $F=0$. Hence, $[I_{\Psi}^{(2)} \cap I_{\Gamma|_L}^{(2)} \cap I_{\Lambda}]_{d-1}=0$, and Claim 1 is proved.

To prove the theorem we need to prove the following equality
$$
H_{R/(I_{\Lambda}^{(2)} \cap I_{\Psi}^{(2)} \cap I_{\Gamma}^{(2)})}(d) = \min \left\{(n+1)r, {n+d \choose n}\right\}.
$$
We now proceed by considering two different cases depending on which of the two possible values the right-hand side may take. Since, by assumption, ${\displaystyle \Big\lfloor \dfrac{1}{n+1}\binom{n+d}{n}\Big\rfloor \le r \le \Big\lceil \dfrac{1}{n+1} {n+d \choose n} \Big\rceil}$, it can be easily seen that
\begin{itemize}
\item $ \min \left\{(n+1)r, {n+d \choose n}\right\} = (n+1)r$ holds precisely if ${\displaystyle r= \Big\lfloor \dfrac{1}{n+1} {n+d \choose n} \Big\rfloor}$,
\item $ \min \left\{(n+1)r, {n+d \choose n}\right\}= {n+d \choose n} > r(n+1)$ holds if ${\displaystyle r =\Big\lceil \dfrac{1}{n+1} {n+d \choose n} \Big\rceil > \Big\lfloor \dfrac{1}{n+1} {n+d \choose n} \Big\rfloor}$.
\end{itemize}
The running example of 21 double points in $\PP^3$ falls in the first possibility.

\medskip
\noindent\underline{\textsc{Case 1:}} ${\displaystyle r = \Big\lfloor \dfrac{1}{n+1} {n+d \choose n} \Big\rfloor}$. In this case, ${\displaystyle r(n+1) \le {n+d \choose n}}$, $ nq + \epsilon \leq {n+d-1 \choose n-1}$, 
and by the above observation we need to show that 
$$
H_{R/(I_{\Lambda}^{(2)} \cap I_{\Psi}^{(2)} \cap I_{\Gamma}^{(2)})}(d) = (n+1)r . 
$$
{\bf Claim 2.} $H_{R/(I_{\Lambda}^{(2)} \cap I_{\Psi}^{(2)})  }(d) = e(R/(I_{\Lambda}^{(2)} \cap I_{\Psi}^{(2)})) = (n+1)(r-\epsilon).$ \\
\\
Castelnuovo's inequality (\ref{eq.Castelnuovo}) gives $H_{R/(I_{\Lambda}^{(2)} \cap I_{\Psi}^{(2)})  }(d) \geq
H_{R/(I_{\Lambda} \cap I_{\Psi}^{(2)})}(d-1) + H_{\ovl{R}/I_{\Lambda|_L}^{(2)} }(d)$.
By Claim 1 and Lemma \ref{reduce}(1), $H_{R/(I_{\Psi}^{(2)} \cap I_{\Lambda})  }(d-1) = (n+1)(r-q-\epsilon) + q$. By assumption (ii) and the inequality $nq\leq \binom{n+d-1}{n-1}$, we have
$$H_{\ovl{R}/I_{\Lambda|_L}^{(2)} }(d)=\min\left\{nq, {n-1+d\choose n-1}\right\} = nq = e(\ovl{R}/I_{\Lambda|_L}^{(2)}).$$
Thus, $H_{R/(I_{\Lambda}^{(2)} \cap I_{\Psi}^{(2)})  }(d)\geq  (n+1)(r-q-\epsilon) + q + nq= (n+1)(r-\epsilon)$. Since the other inequality always holds by Corollary \ref{cor.count2}, Claim 2 is proved.

To finish this case it suffices to prove that $I_{\Gamma}^{(2)}$ is multiplicity $[I_{\Lambda}^{(2)} \cap I_{\Psi}^{(2)}]_d$-independent, because then
$$
H_{R/(I_{\Lambda}^{(2)} \cap I_{\Psi}^{(2)} \cap I_{\Gamma}^{(2)})}(d) = H_{R/(I_{\Lambda}^{(2)} \cap I_{\Psi}^{(2)})}(d) + e(R/I_{\Gamma}^{(2)}) = (n+1)(r-\epsilon) + (n+1)\epsilon = (n+1)r.
$$
Instead of proving this statement directly, we will use deformation to consider a family of general points $\Gamma_\t$ having $\Gamma$ as a limit (as we shall explain in the upcoming Step 3).

For now, we observe the following fact. As before, if we add $\epsilon$ general points of $L$ to $2\Lambda|_L$ then we are adding $\epsilon$ general equations to the linear system determined by $[I_{\Lambda|_L}^{(2)}]_d$, and so
	\begin{align}
	H_{\ovl{R}/(I_{\Lambda|_L}^{(2)} \cap I_{\Gamma}) }(d) =  \min\left\{nq + \epsilon, {n-1+d\choose n-1}\right\}=nq+\epsilon. \label{eq.gen2}
	\end{align}
(In our running example, $H_{\ovl{R}/I_{\Lambda|_L}^{(2)} \cap I_{\Gamma}}(d) = (3)(9)+1=28$.)

	\noindent{\bf Step 3.} For $\t = (t_1, \dots, t_\epsilon) \in K^\epsilon$, consider a flat family of general points $\Gamma_\t = \{\gamma_{1,t_1}, \dots, \gamma_{\epsilon,t_\epsilon} \}$ in  $\PP^n$ and a family of hyperplanes $\{L_{t_1}, \dots, L_{t_\epsilon}\}$ such that
	\begin{enumerate}
		\item the point $\gamma_{i, t_i}$ lies in $L_{t_i}$, for all $i = 1, \dots, \epsilon$,
		\item $\gamma_{i,t_i} \not\in L$ for any $t_i \not= 0$ and any $i = 1, \dots, \epsilon$,
		\item $L_0 = L$ and $\gamma_{i,0} = \gamma_i \in L$ for any $i = 1, \dots, \epsilon.$
	\end{enumerate}
	(For the running example, we have a family of general points $\Gamma_t = \{\gamma_t\} \subseteq \PP^3$ and a family of hyperplanes $L_t$, for $t \in K$.)

	\noindent{\bf Step 4.} To prove $I_{\Gamma}^{(2)}$ is multiplicity $[I_{\Lambda}^{(2)} \cap I_{\Psi}^{(2)}]_d$-independent, by Corollary \ref{cor.count2} and Theorem \ref{thm.semicontinuityHF} it suffices to prove that there exists $\t = (t_1, \dots, t_\epsilon) \in K^\epsilon$ such that $I_{\Gamma_\t}^{(2)}$ is multiplicity $[I_{\Lambda}^{(2)} \cap I_{\Psi}^{(2)}]_d$-independent, because then
	$$
(n+1)r \geq H_{R/(I_{\Lambda}^{(2)} \cap I_{\Psi}^{(2)} \cap I_{\Gamma}^{(2)})}(d) \geq H_{R/(I_{\Lambda}^{(2)} \cap I_{\Psi}^{(2)} \cap I_{\Gamma_{\t}}^{(2)})}(d)=(n+1)r.
	$$
	
	Suppose, by contradiction, that such a $\t$ does not exist. Then, by Lemma \ref{lem.reductionCL}, for each $\t = (t_1, \dots, t_\epsilon)$, there exist curvilinear ideals  $J_{i,t_i}$ such that $I_{\gamma_{i,t_i}}^{(2)} \subseteq J_{i,t_i}\subseteq I_{\gamma_{i,t_i}}$ and, by letting $J_{\t} = \bigcap_{i=1}^{\epsilon} J_{i,t_i}$, we then have
	\begin{align}
H_{R/(I_{\Lambda}^{(2)} \cap I_{\Psi}^{(2)} \cap J_\t) }(d)  & < H_{R/(I_{\Lambda}^{(2)} \cap I_{\Psi}^{(2)})}(d) + e(R/J_\t)    = (n+1)(r-\epsilon) + e(R/J_{\t}). \label{eq.gen3} 
	\end{align}
	(For the running example, there is a single point $\gamma_t$, so $I_{\gamma_t}$ is a linear prime and $J_\t$ is a curvilinear ideal $J_t$ with $I_{\gamma_t}^{(2)} \subseteq J_t \subseteq I_{\gamma_t}$ and $H_{R/(I_{\Lambda}^{(2)} \cap I_{\Psi}^{(2)} \cap J_t)}(6) <  (3+1)(20) + e(R/J_t)=80+e(R/J_t)$.)

Since $J_{\t}$ is a curvilinear ideal, by Proposition \ref{prop:dense}, for every $i=1,\ldots,\epsilon$ the family $\{J_{i,t_i}\}$ has a limit $J_{i,0}$.
Let $J_0=\bigcap_{i=1}^{\epsilon} J_{i,0}$.

Let $A:=\{i ~\big|~ \ell\notin J_{i,0}\}$, $B :=\{i ~\big|~ \ell\in J_{i,0}\}$ and $A':=\{i\in A~\big|~\ell \in \sqrt{J_{i,0}}\}$. We set $a=|A|$, $a'=|A'|$, and $b:=|B|$. For each $\t\in K^\epsilon$, set $J_{\t}^A =\bigcap_{i\in A}J_{i,t_i}$, 
and $J_{\t}^B =\bigcap_{i\in A}J_{i,t_i}$, in particular $J_\t=J_\t^A \cap J_\t^B$. We also set $I_{\Gamma}^{A'}= \bigcap_{i\in A'} I_{\gamma_i}$.

	By the semi-continuity of Hilbert function and (\ref{eq.gen3}), there exists an open neighborhood $U$ of $0$ such that for any $\t \in U$, we have the equalities $e(R/J_\t^B)=e(R/J_0^B)$, $e(R/J_\t^A)=e(R/J_0^A)$ (so $e(R/J_\t)=e(R/J_0)$), and 
	\begin{align}
H_{R/(I_{\Lambda}^{(2)} \cap I_{\Psi}^{(2)} \cap J_0^A\cap J_{\t}^B)}(d) =  H_{R/(I_{\Lambda}^{(2)} \cap I_{\Psi}^{(2)} \cap J_{\t}^A\cap J_{\t}^B)}(d)<(n+1)(r-\epsilon) + e(R/J_{\t}), \label{eq.cont}
	\end{align}
	(for the running example, we have  $H_{R/(I_{\Lambda}^{(2)} \cap I_{\Psi}^{(2)} \cap J_0^A\cap J_{\t}^B)}(d) < 82$), and
	$$H_{R/(I_{\Lambda}^{(2)} \cap I_{\Psi}^{(2)} \cap (J_0^A:\ell) \cap J_{\t}^B)}(d-1)= H_{R/(I_{\Lambda}^{(2)} \cap I_{\Psi}^{(2)} \cap (J_0^A:\ell) \cap J_{0}^B)}(d-1).$$

We want to show $H_{R/(I_{\Lambda}^{(2)} \cap I_{\Psi}^{(2)} \cap J_0^A\cap J_{\t}^B)}(d) \geq  (n+1)(r-\epsilon) + e(R/J_{\t})$, which would then contradict (\ref{eq.cont}).
For any $\t \in U$, the Castelnuovo inequality gives
$$
H_{R/I_{\Lambda}^{(2)} \cap I_{\Psi}^{(2)} \cap J_0^A \cap J_{\t}^B}(d) \geq H_{R/I_{\Lambda} \cap I_{\Psi}^{(2)} \cap (J_0^A:\ell) \cap J_{\t}^B}(d-1)
 + H_{\ovl{R}/(I_{\Lambda|_L}^{(2)} \cap I_{\Gamma|_L}^{A'})}(d),
$$
where $\ovl{R}\cong R/(\ell)$ and $I_{\Gamma|_L}^{A'}$ is the defining ideal of $\{\gamma_i\,\mid\,i\in A'\}$ in $\ovl{R}$.

We examine the first summand which, by the choice of $U$, equals $H_{R/(I_{\Lambda}^{(2)} \cap I_{\Psi}^{(2)} \cap (J_0^A:\ell) \cap J_{0}^B)}(d-1)$ for every $\t\in U$. By Claim 1, the ideal $I_{\Lambda} \cap I_{\Psi}^{(2)} \cap I_{\Gamma|_L}^{(2)} $ is multiplicity $(d-1)$-independent, so -- by Lemma \ref{multind} -- 
 the larger ideal $I_{\Lambda} \cap I_{\Psi}^{(2)} \cap (J_0^A:\ell) \cap J_0^B$ is multiplicity $(d-1)$-independent too. Thus,
\begin{align}
H_{R/(I_{\Lambda} \cap I_{\Psi}^{(2)} \cap (J_0^A:\ell) \cap J_{0}^B)}(d-1) & = e(R/(I_{\Lambda} \cap I_{\Psi}^{(2)} \cap (J_0^A:\ell) \cap J_{0}^B))  \label{eq.gen50} \\
	& = e(R/(I_{\Lambda} \cap I_{\Psi}^{(2)})) + e(R/(J_0^A:\ell) \cap J_0^B)\nonumber \\
	& = q+(n+1)(r-q-\epsilon) + e(R/(J_0^A:\ell) \cap J_0^B) \nonumber \\
	& = q+(n+1)(r-q-\epsilon) + e(R/J_0) - a', \nonumber
	\end{align}
where $a'$ is the cardinality of $\{i \in A\,\mid\, \ell \text{ is not regular on }R/J_{i,0}\}$. The last equality holds because $e(R/J_0) = e(R/J_0:\ell) + e(R/(J_0,\ell)) = e(R/J_0^A:\ell) + e(R/(J_0,\ell)) = e(R/(J_0^A:\ell )) + e(R/J_0^B) + a'$.

Now, the inclusion $2\Lambda_{|L} \cup \{\gamma^i ~\big|~ i \in F\} \subseteq 2\Lambda_{|L} \cup \Gamma$ and (\ref{eq.gen2}) allow the use of Lemma \ref{reduce}(1) to deduce that
$$
H_{\ovl{R}/(I_{\Lambda|_L^{(2)}}\cap I_{\Gamma|_L}^{A'})}(d) \geq e\left(\ovl{R}/\left(I_{\Lambda|_L^{(2)}}\cap I_{\Gamma|_L}^{A'}\right)\right) = nq + a'.
$$
Putting all these together, we obtain that, for any $\t \in U$,
	\begin{align*}
H_{R/I_{\Lambda}^{(2)} \cap I_{\Psi}^{(2)} \cap J_0^A \cap J_{\t}^B}(d) &  \ge H_{R/I_{\Lambda} \cap I_{\Psi}^{(2)} \cap (J_0^A:\ell) \cap J_{\t}^B}(d-1)  + H_{\ovl{R}/I_{\Lambda|_L}^{(2)} \cap I_{\Gamma}^{A'}}(d)\\
& \ge q+(n+1)(r-q-\epsilon)+e(R/J_0) - a' + (nq+a') \\
	& = (n+1)(r-\epsilon) + e(R/J_0)  = (n+1)(r-\epsilon) + e(R/J_{\t}).
	\end{align*}
(For the running example, this implies one of the following inequalities $H_{R/(I_{\Lambda}^{(2)} \cap I_{\Psi}^{(2)} \cap J_0)}(6) \geq 80 + e(R/J_t)$ or $H_{R/(I_{\Lambda}^{(2)} \cap I_{\Psi}^{(2)} \cap J_t)}(6) \geq 80 + e(R/J_t)$.) This is a contradiction to (\ref{eq.gen3}), and we are done.

\medskip
\noindent \underline{\textsc{Case 2:}} ${\displaystyle r > \Big\lfloor \dfrac{1}{n+1} {n+d \choose n} \Big\rfloor}$. In this case, ${\displaystyle r = \Big\lceil \dfrac{1}{n+1}{n+d \choose n}\Big\rceil}$,  and  we have
$${\displaystyle r(n+1) > {n+d \choose n} \text{ and } nq+\epsilon > {n+d-1 \choose n-1}.}$$

First, one considers the case where $nq \ge {n+d-1 \choose n-1}$. Then by (i), we have
	$$
	H_{\ovl{R}/ I_{\Lambda|_L}^{(2)}}(d) =  \min \left\{nq, {n+d-1 \choose n-1}\right\} = {n+d-1 \choose n-1}.$$	
	On the other hand, by (ii), we have
	$$
	H_{R/(I_{\Gamma}^{(2)}\cap I_{\Psi}^{(2)})}(d-1) = \min \left\{ (n+1)(r-q), {n+d-1 \choose n}\right\}.$$
For any set $\Delta$ of $q-\epsilon$ general points in $\PP^n$ one has
$H_{R/(I_{\Delta} \cap I_{\Gamma}^{(2)}\cap I_{\Psi}^{(2)})}(d-1)= {n+d-1 \choose n} = e(R/(I_{\Delta} \cap I_{\Gamma}^{(2)}\cap I_{\Psi}^{(2)}))$. As in the proof of Claim 1, this yields that for any subset $\Lambda'\subseteq \Lambda$ consisting of $q-\epsilon$ points one has
$H_{R/(I_{\Gamma'} \cap I_{\Gamma}^{(2)}\cap I_{\Psi}^{(2)})}(d-1)= {n+d-1 \choose n}.$
By Lemma \ref{reduce}(2), one obtains
	$
	H_{R/(I_{\Lambda} \cap I_{\Gamma}^{(2)}\cap I_{\Psi}^{(2)})}(d-1) = {n+d-1 \choose n}.$
	Thus, by the Castelnuovo inequality, we get
	\begin{align*}
	H_{R/(I_{\Lambda}^{(2)}\cap  I_{\Gamma}^{(2)}\cap I_{\Psi}^{(2)})}(d) & \ge H_{R/(I_{\Lambda} \cap I_{\Gamma}^{(2)}\cap I_{\Psi}^{(2)})}(d-1) + H_{\ovl{R}/ I_{\Lambda|_L}^{(2)}}(d)\\
	& = {n+d-1 \choose n} + {n+d-1 \choose n-1} = {n+d \choose n}.
	\end{align*}
	Hence, the desired equality holds and $2\Lambda \cup 2\Psi \cup 2\Gamma$ satisfies $\text{AH}_{n,d}$.
	 	
	We may now assume that ${\displaystyle 0 < \nu := {n+d-1 \choose n-1} - nq < \epsilon}$, and let $\Gamma' = \{\gamma_1, \dots, \gamma_{\nu}\}\subseteq \Gamma$. By a similar argument (or similar to the proof of Claim 1), it can be shown that
	$$
	H_{\ovl{R}/ (I_{\Lambda|_L}^{(2)} \cap I_{\Gamma'|L})}(d) = {n+d-1 \choose n-1} = nq + \nu.
	$$
	Thus, by Lemma \ref{reduce}(2), one has
	$$	H_{\ovl{R}/ (I_{\Lambda|_L}^{(2)} \cap I_{\Gamma|L})}(d) =  \min \left\{nq + \epsilon, {n+d-1 \choose n-1}\right\} = {n+d-1 \choose n-1}.$$

To show that $2\Lambda \cup 2\Psi \cup 2\Gamma$ satisfies $\text{AH}_{n,d}$ we need to prove
$$H_{R/(I_{\Lambda}^{(2)}\cap  I_{\Gamma}^{(2)}\cap I_{\Psi}^{(2)})}(d) = \binom{n+d}{n}.$$
Let $\t$ and $\Gamma_\t$ be defined as in Case 1.
By the semi-continuity of the Hilbert function, there exists a neighborhood $U$ of $0$ such that for $\t \in U$ we have
$$H_{R/(I_{\Lambda}^{(2)} \cap I_{\Psi}^{(2)} \cap I_{\Gamma})}(d) = H_{R/(I_{\Lambda}^{(2)} \cap I_{\Psi}^{(2)} \cap I_{\Gamma_\t} ) }(d).$$

\noindent{\bf Claim 3.} To finish the proof it suffices to find an ideal $K\supseteq  I_{\Gamma_\t}^{(2)}$ such that $K$ is multiplicity $[I_{\Lambda}^{(2)} \cap I_{\Psi}^{(2)}]_d$-independent, and $e(R/K)=n \epsilon + \nu$.

Assume such an ideal $K$ exists. The first assumption gives $H_{R/(I_{\Lambda}^{(2)} \cap I_{\Psi}^{(2)} \cap K)}(d) = H_{R/(I_{\Lambda}^{(2)} \cap I_{\Psi}^{(2)})}(d) + e(R/K)$.
Also, by Claim 2, $H_{R/(I_{\Lambda}^{(2)} \cap I_{\Psi}^{(2)})  }(d) = (n+1)(r-\epsilon)$. Finally, $e(R/K)$ is precisely the amount needed to ensure that $H_{R/(I_{\Lambda}^{(2)} \cap I_{\Psi}^{(2)} \cap K)}(d)=\binom{n+d}{n}$, because one has
$$	\begin{array}{lll}
H_{R/(I_{\Lambda}^{(2)} \cap I_{\Psi}^{(2)} \cap K)}(d) & =(n+1)(r-\epsilon) + e(R/K)	& =(n+1)(r-\epsilon) + n\epsilon + \nu \\
		& =(n+1)r - \epsilon + \nu 	& = (n+1)r - (nq+\epsilon) + (nq+\nu)\\
		& = {n+d-1 \choose n} + {n+d-1\choose n-1} & = {n+d\choose n}.
\end{array}
$$

Now,
$$H_{R/(I_{\Lambda}^{(2)} \cap I_{\Psi}^{(2)} \cap I_{\Gamma})}(d) = H_{R/(I_{\Lambda}^{(2)} \cap I_{\Psi}^{(2)} \cap I_{\Gamma_\t} ) }(d)={n+d\choose n},$$
where the rightmost equality follows from Lemma \ref{reduce}(2). This proves Claim 3.

Finally, it is easily seen that $K:=I_{\Gamma'}^{(2)} \cap I_{(\Gamma - \Gamma')|_L}$ satisfies the desired properties. This concludes the proof of the theorem.
\end{proof}

Modulo the exceptional cases, which are considered in the following sections, we now give a complete proof of the Alexander--Hirschowitz theorem.

\begin{Theorem}(Alexander--Hirschowitz)
	\label{thm.AHcore}
For every $n\geq 1$ and $d\geq 1$, a set $X$ of $r$ general double points in $\PP_{\CC}^n$ is $\text{AH}_{n}(d)$, with the following exceptions:
\begin{enumerate}
\item $d=2$ and $2\leq r \leq n$;
\item $d=3$, $n=4$ and $r=7$;
\item $d=4$, $2\leq n \leq 4$ and $r=\binom{n+2}{2}-1$.
\end{enumerate}
\end{Theorem}

\begin{proof} By Remarks \ref{r=1} and \ref{d=1}, we may assume that $r\geq 2$ and $d\geq 2$.
	The statement for $n=1$ is proved in Proposition \ref{P1}. The case where $n=2$ is treated in Section \ref{sec.P2}. Thus, we may also assume that $n\geq 3$. The exceptional cases are discussed in Sections \ref{exceptional},\ref{sec.P2} and \ref{sec.cubic}. Furthermore, it will be shown that for fixed $d$ and $n$, the given value of $r$ is the only exceptional case of $r$ general double points not being $\text{AH}_n(d)$. Finally, for $n$ and $d$ not in the list of exceptional cases, by Lemma \ref{rmk.Hsemi}, we only need to consider values of $r$ such that
	$$\left\lfloor \dfrac{1}{n+1}{n+d \choose d}\right\rfloor \le r \le \left\lceil \dfrac{1}{n+1}{n+d \choose d}\right\rceil.$$
	
Our argument proceeds by considering small values of $d$ and then using induction together with Theorem \ref{core}. The statement for $d = 2$ is proved in Lemma \ref{Exd=2}.
The statement for $d=3$ is examined in Section \ref{sec.cubic}. Therefore, we may assume now that $n\geq 3$ and $d\geq 4$.

We will use induction on $n$ to prove the assertion for $d=4$. Note that the statement for $d=4$ and $3 \le n \le 4$ is proved in Lemma \ref{Exd=4}. On the other hand, if the statement has been shown for $5 \le n \le 7$, then Theorem \ref{core} applies to prove the desired assertion for all $n\geq 8$ too. This is because condition (i) holds by the induction hypothesis on $n$, condition (ii) holds as shown in Section \ref{sec.cubic}, and condition (iii) holds because, for $n\geq 8$, by Lemma \ref{lem.numeric}(3) we have $r-q-\epsilon \geq n+1$, and thus $\text{AH}_{r-q-\epsilon}(2)$ holds as shown in Lemma \ref{Exd=2}. It remains to consider $d=4$ and $5 \le n \le 7$. We shall leave this case until later in the proof.

In general, for $d \ge 5$, the proof proceeds by a double induction on $d$ and $n$. Observe that if the statement has been proved for $d =5, 6$ and $n = 3,4$, then Theorem \ref{core} applies to prove the statement for all $d \ge 5$ and $n \ge 3$. Therefore, we only need to establish the desired assertion for $d=5,6$ and $n=3,4$.

We conclude the proof by analyzing the needed cases, i.e. when $d=4$ and $5 \le n \le 7$, or when $d=5,6$ and $n = 3,4$. Most cases are also proved by applying Theorem \ref{core}.

\noindent{\it Case 1: $d=4, n = 5$.} In this case, we need to consider $r = 21$ general double points in $\PP^5$, $q = 14$ and $\epsilon = 0$.  Direct Macaulay 2 \cite{M2} computation can be used to verify that the assertion holds.

\noindent{\it Case 2: $d=4, n=6$.} In this case, we need to consider $r = 30$ general double points in $\PP^6$, $q = 21$ and $\epsilon = 0$. Theorem \ref{core} applies because 21 general double points are $\text{AH}_5(4)$ by Case 1, and 9 general double points are $\text{AH}_6(3)$ (as shown in Section \ref{sec.cubic}) and $\text{AH}_6(2)$ (by Lemma \ref{Exd=2}).

\noindent{\it Case 3: $d=4, n = 7$.} In this case, we need to consider $r = 41$ or $42$ general double points in $\PP^7$, $q = 29$ or $30$ and $\epsilon = 5$ or $6$. Direct Macaulay 2 \cite{M2} computation shows that 41 and 42 general double points in $\PP^7$ are indeed $\text{AH}_7(4)$.

\noindent{\it Case 4: $d=5, n=3$.} In this case, we need to consider $r = 14$ general double points in $\PP^3$, $q = 7$ and $\epsilon = 0$. Theorem \ref{core} applies because 7 general double points are $\text{AH}_2(5)$ (by Theorem \ref{P2}), $\text{AH}_3(4)$ (by Lemma \ref{Exd=4}), and $\text{AH}_3(3)$ (as shown in Section \ref{sec.cubic}).

\noindent{\it Case 5: $d=5, n=4$.} In this case, we need to consider $r = 25$ or $26$ general double points in $\PP^4$, $q = 13$ or $q = 15$, and $\epsilon = 3$ or $0$. For $r = 25, q = 13$ and $\epsilon = 3$, Theorem \ref{core} applies because 13 general double points are $\text{AH}_3(5)$ by Case 4, 12 general double points are $\text{AH}_4(4)$ (by Lemma \ref{Exd=4}), and 9 general double points are $\text{AH}_4(3)$ (as shown in Section \ref{sec.cubic}). For $r = 26, q = 15$ and $\epsilon = 0$, Theorem \ref{core} applies because $15$ general double points are $\text{AH}_3(5)$ by Case 4, 11 general double points are $\text{AH}_4(4)$ (by Lemma \ref{Exd=4}) and $\text{AH}_4(3)$ (as shown in Section \ref{sec.cubic}).

\noindent{\it Case 6: $d = 6, n = 3$.} In this case, we need to consider $r = 21$ general double points in $\PP^3$, $q = 9$ and $\epsilon = 1$. Theorem \ref{core} applies because $9$ general double points are $\text{AH}_2(6)$ (by Theorem \ref{P2}), 12 general double points are $\text{AH}_3(5)$ by Case 4, and $11$ general double points are $\text{AH}_3(4)$ (by Lemma \ref{Exd=4}).

\noindent{\it Case 7: $d = 6, n = 4$.} In this case, we need to consider $r = 42$ general double points in $\PP^4$, $q = 21$ and $\epsilon = 0$. Theorem \ref{core} applies because $21$ general double points are $\text{AH}_3(6)$ by Case 6, and 21 general double points are $\text{AH}_4(5)$ by Case 5 and $\text{AH}_4(4)$ (by Lemma \ref{Exd=4}).
\end{proof}


\section{The exceptional cases}\label{exceptional}

In this section, we consider the exceptional cases listed in Theorem \ref{AH} and show that they are indeed the only exceptional cases for given $n$ and $d$.
We begin by considering the case where $d = 2$.

\begin{Lemma}\label{Exd=2}
	A set of $r\geq 1$ general double points in $\PP^n$ is not $\text{AH}_n(2)$ if and only if $2\leq r\leq n$.
\end{Lemma}

\begin{proof} A single double point is $\text{AH}_n(2)$ (e.g. by Remark \ref{r=1}), so we may assume $r\geq 2$. First we prove that a set of $r\geq n+1$ general double points in $\PP^n$ is $\text{AH}_n(2)$.
 Let $Y=\{P_1,\dots,P_r\}$ denote a set of $r\geq n+1$ general points in $\PP^n$ and let $X = 2Y$. It is easily seen that $X$ is $\text{AH}_n(2)$ if and only if $I_X$ contains no quadrics. 
 
By Lemma \ref{reduce}, it suffices to show that $I_X$ contains no quadrics when $r=n+1$. When $r=n+1$, by a change of variables, we can assume that $P_i$ is the $i$-th coordinate point, for $i = 1, \dots, n+1$. That is, $P_i = [0:\dots:0:1:0:\dots:0]$, where the value 1 appears at the $i$-th position.  In this case, $I_Y$ is the squarefree monomial ideal
	$$I_Y=\pp_0\cap \ldots \cap \pp_n=(x_ix_j\,\mid\, 0\leq i < j \leq n)$$
	where $\pp_i=(x_j\,\mid\,0\leq j \leq n, \,j\neq i)$ for every $i=0,\ldots,n$. It is well-known that $I_X = I_Y^{(2)} = (x_ix_jx_h\,\mid\, 0\leq i < j<h \leq n)$ (e.g. \cite[Cor. 3.8]{GGSV}, or \cite[Cor.~4.15(a)]{Ma}). Thus, $I_X$ indeed contains no quadrics.

To conclude the proof we need to show that any set $X$ of $2\leq r \leq n$ general double points in $\PP^n$ is not $\text{AH}_n(2)$. Since $r\leq n$, we may assume that $P_i$ is the $i$-th coordinate point for $1 \leq i \leq r$. We first claim that $I_X$ contains precisely $\binom{n-r+2}{2}$ linearly independent quadrics. Indeed, again, let $\pp_i$ be the defining ideal of $P_i$, for $i = 1, \dots, r$. It is easy to see that $(x_{r},x_{r+1},\ldots,x_n)\subseteq \pp_i$ for all $i=1,\dots,r$. Thus, $(x_{r},\ldots,x_n)^2\subseteq \bigcap_{i=1}^r \pp_i^2=I_X$. By modularity law, it follows that
	$$
	I_{X} = \left(I_{r-1,r}\right)^{(2)} + (x_{r},\ldots,x_n)^2
	$$
	where $\left(I_{r-1,r}\right)^{(2)}=\bigcap_{0\leq j_1<j_2<\ldots, < j_{r-1}\leq r-1} (x_{j_1},\ldots,x_{j_{r-1}})^2$ (this is called the second symbolic power of the star configuration of codimension $r-1$ in the variables $x_0,\ldots,x_{r-1}$). It is known that  $I_{r-1,r}$ is generated by all squarefree quadrics in $x_0,\ldots,x_{r-1}$ (e.g. \cite[Thm~2.3]{PS}), and $\left(I_{r-1,r}\right)^{(2)}$ is generated in degree 3 and higher (see e.g. \cite[Cor.~3.8]{GGSV}). It follows that the quadrics in $I_X$ are precisely the $\binom{n-r+2}{2}$ quadrics contained in $(x_{r+1},\ldots,x_n)^2$, proving the claim.

Now, our claim on $I_X(2)$ implies that $X$ is $\text{AH}_n(2)$ only if
$$ {n-r+2 \choose 2} = {n+2 \choose 2} - H_{R/I_X}(2) = \max \left\{0, {n+2 \choose 2} - r(n+1)\right\}.$$
Since ${n-r+2 \choose 2} > 0$, this is only possible if ${n-r+2 \choose 2} = {n+2 \choose 2} - r(n+1)$, which implies $r^2-r = 0$, and thus gives a contradiction. Therefore, $X$ is not $\text{AH}_n(2).$
\end{proof}

We continue with the cases $d=4$ and $2 \le n \le 4$.

\begin{Lemma}\label{Exd=4}
Suppose that $2 \le n \le 4$. Then, a set of $r$ general double points in $\PP^n$ is not $\text{AH}_n(4)$ if and only if $r = {n+2 \choose 2} -1.$
\end{Lemma}

\begin{proof}
Let $Y = \{P_1, \dots, P_r\}$ be a set of $r$ general points in $\PP^n$ and let $X = 2Y$. We shall first show that for $r = {n+2 \choose 2}-1$, $X$ is not $\text{AH}_n(4)$. Indeed, since $r < {n+2 \choose 2}$, $I_Y$ contains a nonzero quadric, say $Q$. Then, $Q^2$ is a nonzero quartic in $I_Y^2 \subseteq I_Y^{(2)} = I_X$. This implies that $H_{R/I_X}(4) \le {n+4 \choose 4} - 1$. It is easy to check that for $2 \le n \le 4$, ${n+4 \choose 4} - 1 < \left[{n+2 \choose 2} - 1\right](n+1) = r(n+1)$. Therefore, $X$ is not $\text{AH}_n(4)$.
	
We shall now show that $r = {n+2 \choose 2} - 1$ is indeed the only exceptional case. The statement for $n = 2$ is proved in Theorem \ref{P2}. Suppose that $3 \le n \le 4$.

For $n=3$, by Corollary \ref{rpts}, it suffices to prove that a set of 8 general double points and a set of 10 general double points in $\PP^3$ are both $\text{AH}_3(4)$. Similarly, for $n=4$, it suffices to establish $\text{AH}_4(4)$ property for a set of 13 general double points and a set of 15 general double points in $\PP^4$.

\noindent\underline{$n=3$ and $r=8$.} Observe that $\left\lfloor \frac{1}{3+1}{4+3\choose 3}\right\rfloor=8=r$, so Theorem \ref{core} applies if its hypotheses are satisfied. In this case, we have $q=4$ and $\epsilon=0$. Thus, condition (i) holds because $4$ general double points in $\PP^2$ are $\text{AH}_2(4)$ (by Theorem \ref{P2}), and condition (iii) holds because $4$ general double points are $\text{AH}_3(2)$ (by Lemma \ref{Exd=2}). To prove that condition (ii) holds, we need to show that $4$ general double points are $\text{AH}_3(3)$. This follows from Section \ref{sec.cubic}.

We can also prove this statement directly by considering the 4 coordinate points in $\PP^3$. Let $I$ be the defining ideal of these coordinate points. Then, $I=(x_ix_j\,\mid\,0\leq i <j \leq 3)$, and it can be checked that $I^{(2)}$ is minimally generated by the four squarefree monomials of degree 3. In particular, $H_{R/I^{(2)}}(3)=16$ which is the expected dimension, so condition (ii) of Theorem \ref{core} holds.

In the remaining 3 cases, i.e. when $n=3$ and $r=10$, or when $n=4$ and $r=13$ or $15$ we cannot apply Theorem \ref{core} because $r$ is not one of the two possible values needed to apply the theorem. We will instead use Theorem \ref{thm.Terr}.

\noindent\underline{$n=3$ and $r=10$.} We shall apply Theorem \ref{thm.Terr} for $q=6$. Clearly, $6$ general double points is $\text{AH}_2(4)$ (by Theorem \ref{P2}). Thus, it remains to show that the union of 4 general double points and 6 general simple points on a hyperplane is $\text{AH}_3(3)$.

Let $Y_1$ be the set of the four coordinate points in $\PP^3$. As shown above, we have $$H_{R/I_{Y_1}^{(2)}}(2)=10\qquad \text{ and }\qquad H_{R/I_{Y_1}^{(2)}}(3)=16.$$
Let $L$ be a hyperplane not containing any point of $Y_1$. By taking $I=I_{Y_1}^{(2)}$, Proposition \ref{prop.hyperp}(2) holds for any $u$ satisfying
$$
H_{R/I}(3) + u \leq H_{R/I}(2) + {3+3-1\choose 3-1},
$$
i.e., whenever $16 + u \leq 10 + 10$, i.e., $u\leq 4$.
Therefore, if we let $Y_0$ be a set of $u=4$ general points on $L$, then $I_{Y_1}^{(2)}\cap I_{Y_0}$ does not contain any cubic. Now, let $Y_2$ be obtained by adding two points to $Y_0$, then $I_{Y_1}^{(2)}\cap I_{Y_2} \subseteq I_{Y_1}^{(2)} \cap I_{Y_0}$ contains no cubics. That is, $2Y_1 \cup Y_2$ is $\text{AH}_3(3)$.

\noindent\underline{$n=4$ and $r=13$.} We shall apply Theorem \ref{thm.Terr} for $q = 8$. So one may take $Y_1$ to be the set of the 5 coordinate points of $\PP^4$ and $L$ to be a hyperplane not containing any of these points. Then $I_{Y_1}$ is again generated by all squarefree monomials of degree 2 in $R$, and $I_{Y_1}^{(2)}$ by the squarefree monomials of degree 3. It follows that $Y_1$ is $\text{AH}_4(3)$, and in particular $H_{R/I_{Y_1}^{(2)}}(3) = 25$. Then inequality (2) of Proposition \ref{prop.hyperp} then becomes $25 + q \leq 15 + 20$, so if we add 10 general simple points in $L$ to $2Y_1$ we obtain a scheme containing no cubics.

In particular, if we take $Y_2$ to be a set of $q=8$ general points on $L$, then assumption (2) of Theorem \ref{thm.Terr} is satisfied, so $Y_1\cup Y_2$ is a set of 13 points in $\PP^4$ which is $\text{AH}_4(4)$. By Lemma \ref{special} any set of 13 general points is $\text{AH}_4(4)$. 

\noindent\underline{$n=4$ and $r=15$.} We shall apply Theorem \ref{thm.Terr} for $q = 10$. Clearly, a set of $q=10$ general double points is $\text{AH}_3(4)$ as shown above. 
Thus, it suffices to show that the union of 5 general double points and 10 general simple points in a hyperplane is $\text{AH}_4(3)$. This follows by the same argument of the previous case.
\end{proof}

We conclude this section with the case where $d=3$ and $n=4$.

\begin{Lemma}\label{Exr=7}
	A set of $r$ general double points in $\PP^4$ is $\text{AH}_4(3)$ if and only if $r \not= 7$.
\end{Lemma}

\begin{proof} We first prove that a set of 7 general double points in $\PP^4$ is not $\text{AH}_4(3)$. Let $Y=\{P_1,\ldots,P_7\}\subseteq \PP^4$ be a set of 7 general points, a simple computation shows that $2Y$ is $\text{AH}_4(3)$ if and only if $I_X^{(2)}$ contains no non-zero cubic.

By a result of Castelnuovo (e.g. \cite[Thm~1]{EH1}), given any set of $t+3$ points in general position in $\PP^t$, there exists a unique rational normal curve $C_t$ passing through all of them, whose equations are given by the $2 \times 2$ minors of a \emph{1-generic} matrix. In particular, there is a (unique) rational normal curve $C_4$ passing through our 7 points in $\PP^4$, whose equation, in an appropriate coordinate system,  is
$$
I:=I_2 \begin{pmatrix}
x_0 & x_1 & x_2 & x_3\\
x_1 & x_2 & x_3 & x_4
\end{pmatrix}.
$$

One can check directly that $I^{(2)}$ contains (precisely) one cubic, namely
$$
x_2^3 - 2x_1x_2x_3 + x_0x_3^2 + x_1^2x_4 - x_0x_2x_4.
$$
Thus, a set of 7 general double points in $\PP^4$ is not $\text{AH}_4(3)$.

Alternatively, it is also known that
$I=I_2 \begin{pmatrix}
x_0 & x_1 & x_2\\
x_1 & x_2 & x_3 \\
x_2 & x_3 & x_4
\end{pmatrix}$ and it can be seen that $f=\det\begin{pmatrix}
x_0 & x_1 & x_2\\
x_1 & x_2 & x_3 \\
x_2 & x_3 & x_4
\end{pmatrix}$ is singular at all points of $C_4$.

By Corollary \ref{rpts}, to conclude it suffices to show that sets of $r=6$ and $r=8$ general double points in $\PP^4$ are $\text{AH}_4(3)$. As Theorem \ref{core} could only be applied if $r=7$, then we invoke Theorem \ref{thm.Terr} in both cases. First, observe that by Lemma \ref{Exd=2}, sets of 5 general double points in $\PP^4$ are $\text{AH}_4(2)$. If $r=6$, to apply Theorem \ref{thm.Terr} we need $q$ with $15\leq 4q\leq 18$, thus $q=4$. Then, assumption (1) holds for the reasons stated in the proof of Lemma \ref{Exd=4} (the case where $n=3$ and $r=8$), and (2) holds because $r-q=2$ general double points are $\text{AH}_4(2)$ (because 5 double coordinate points are, and because of Lemma \ref{reduce}(1)) and by Proposition \ref{prop.hyperp} (we need to add $u=q=4$ general points to the two double points).

The case $r=8$ is proved similarly. In this case, one may take $q$ satisfying $15\leq 4q \leq 25$. If we take $q=4$ then, as above, assumption (1) of Theorem \ref{thm.Terr} is satisfied. For assumption (2), we need to prove there exists no quadric through a set $Z$ of 4 general double points and 4 general simple points. However it is easily seen that the only quadric through 4 general double points in $\PP^4$ is the square of the hyperplane containing them. Since the remaining 4 simple points are general, we may take them outside this hyperplane, so there is no quadric in $I_Z$. 

An application of Theorem \ref{thm.Terr} now finishes the proof.
\end{proof}

We end this section by noting that the case of cubics, i.e., when $d=3$, for an arbitrary value of $n$ is much more subtle. Section \ref{sec.cubic} is devoted to handle this case.


\section{The case of $\PP^2$ ($n=2$)} \label{sec.P2}

This section focuses on the double points in $\PP^2$. Particularly, we shall identify all exceptional cases when $n = 2$. While one could prove this case with more elementary arguments, we have chosen to employ Theorem \ref{core} to provide the reader with a further illustration of its application.

\begin{Theorem}\label{P2}
Let $X$ be any set of $r$ general points in $\PP^2$.
Then $2X$ is $\text{AH}_2(d)$ for every $d\geq 1$, except for the exceptional cases of $r=2$ and $d=2$, and $r=5$ and $d=4$.
\end{Theorem}

\begin{proof} Let $R = \CC[x,y,z]$ be the homogeneous coordinate ring of $\PP^2$. We shall consider different cases based on the values of $d$.
	
\noindent\textit{Case 1: $d=1$.} It suffices to prove the assertion for $r = 1$ since the degree of a double point in $\PP^3$ is $3 = H_R(1)$. This case follows from Remark \ref{r=1}.

\noindent\textit{Case 2: $d=2$.} The assertion is true for $r =1$ by Remark \ref{r=1}. The case where $r = 2$ is an exceptional case by Lemma \ref{Exd=2}. Suppose that $r \ge 3$.  Since the degree of 3 double points in $\PP^2$ is 9, which is bigger than $6 =H_R(2)$, then $2X$ is $AH_2(2)$ if and only if there exists no conic in $\PP^2$ it suffices to prove that there is no conic in $\PP^2$ with $r$ double points. Clearly, it suffices to prove it when $r=3$. By B\'ezout theorem, the equation of every conic with 3 double points is divisible by the equations of the three lines connecting 2 of these points -- this gives a contradiction.

\noindent\textit{Case 3: $d=3$.} The statement is true for $r = 1$, again by Remark \ref{r=1}. When $r=2$ we need to show that $H_{R/I_X^{(2)}}(3)=4$. Observe that by B\'ezout theorem, a cubic with 2 double points must contain the line connecting these points. That is, this cubic factors as a line and a conic going through these 2 points. Since the Hilbert function of 2 general points in $\PP^2$ is $1, 2, 2, \dots$, it follows that the space of conic going through these 2 points has dimension 4. Particularly, the space of cubic with 2 double points has dimension 4.
Thus, the assertion is true for $r = 2$.

Observe further that by B\'ezout theorem, a cubic with 3 double points must contain 3 lines connecting 2 of these points, and so there is a unique such cubic, which is the union of the 3 lines. It follows that $H_{I_X^{(2)}}(3) = 10-1=9=e(R/I_X^{(2)})$, therefore, the assertion is true for $r = 3$.

Suppose that $r \ge 4$. Since the degree of 4 double points is $12 > 10 = H_R(3)$, it suffices to show that there is no cubic containing 4 double points. By B\'ezout theorem again, if such a cubic existed then it would contain the 6 lines connecting any 2 of these 4 points, a contradiction.

\noindent\textit{Case 4: $d \ge 4$.} Recall that, from Theorem \ref{core}, a set of $r$ general double points in $\PP^2$ with $\left\lfloor \frac{1}{3}{d+2\choose 2}\right\rfloor \leq r \leq \left\lceil \frac{1}{3}{d+2\choose 2}\right\rceil$ is $\text{AH}_2(d)$ if
\begin{enumerate}
	\item $q$ general double point in $\PP^1$ is $\text{AH}_1(d)$ (which holds by Proposition \ref{P1}),
	\item $r-q$ general double points in $\PP^2$ are $\text{AH}_2(d-1)$, and
	\item $r-q-\epsilon$ general double points in $\PP^2$ are $\text{AH}_2(d-2)$,
\end{enumerate}
where  $q \in \mathbb N_0$ and $0 \le \epsilon \le 1$ are such that $2q+\epsilon = 3r - {d+1 \choose 2}.$

When $d=4$, Remark \ref{rmk.Hsemi} says there are no exceptions if the case $r=\frac{1}{3}{4+2 \choose 2}=5$ is not an exceptional case. However, it is an exceptional case (and in fact in this case $q=2$, $\epsilon=1$, so condition (3) of Theorem \ref{core} is not satisfied -- because it is the exceptional case of $2$ double points in degree 2). By Lemma \ref{reduce}, we need to show that the cases $r=4,6$ are not exceptional cases.

When $r=4$, the first numerical condition in Theorem \ref{thm.Terr} is $2\leq 2q \leq 5$, so $1\leq q \leq 2$. Taking $q=1$, assumption (1) of Theorem \ref{thm.Terr} is satisfied by Proposition \ref{P1}. On the other hand, a set of $r-q=3$ general double points in $\PP^2$ is $AH_2(3)$ by the above and there is precisely one cubic passing through all the three points twice. So there is no cubic passing through them twice and passing through an additional general simple point (which we can take to be outside the cubic). Therefore, assumption (2) is satisfied too, and this case follows by Theorem \ref{thm.Terr}.  

When $r=6$ the proof is very similar. The second numerical condition in in Theorem \ref{thm.Terr} is $5\leq 2q \leq 8$, so $3\leq q \leq 4$.
We take $q=3$ so again we have $r-q=3$, and then assumptions (1) and (2) of Theorem \ref{thm.Terr} are satisfied as above, thus proving that $r=6$ is not an exceptional case, and concluding the case $d=4$. 

For $d = 5$, by Remark \ref{rmk.Hsemi} it suffices to prove that a set of $r=7$ general double points is $AH_2(5)$. In this case Theorem \ref{core} applies, because the induction hypotheses (1)--(3) are satisfied with the only possible exception of (2) when $r-q=5$ (as it reduces to the exceptional case of 5 double points in degree 4), i.e. $q=2$. Since $2q+\epsilon = 3r - 15$, then $\epsilon = 2$, which is a contradiction. 


For $d = 6$, by Remark \ref{rmk.Hsemi} we need to prove that sets of $r=9,10$ general double points in $\PP^2$ are $AH_2(6)$. The induction hypotheses (1)--(3) of Theorem \ref{core} are satisfied except possibly assumption (3) when $r-q-\epsilon = 5$ (in this case (3) reduces to the exceptional case of 5 double points in degree 4). Since $2q+\epsilon = 3r-21$, we get $q = 2r - 16$ and $\epsilon = 11 - r$. Since $r \leq 10$ and $0 \le \epsilon \le 1$, we must have $r = 10$, $\epsilon = 1$ and $q = 4$. This particularly shows that Theorem \ref{core} applies when $r=9$, so the case $r=9$ and $d=6$ is not an exceptional case. As a consequence, there is a unique sextic containing 9 general double points (since the degree of 9 double points is 27). On the other hand, the Hilbert function of 9 general points is $1, 3, 6, 9, 9, \dots$, and so there is only one cubic passing through 9 general points. Thus, the unique sextic with 9 general double points is the double cubic passing through these 9 general points. As the remaining point is general, we can take it outside the sextic, resulting in no sextic passing through 10 general double points.

Since there are no exceptional cases in degrees 5 and 6, by Theorem \ref{core}, we conclude that there is no exceptional cases in any degree $d \ge 5$, finishing the proof.

\end{proof}


\section{The case of cubics ($d=3$)}\label{sec.cubic}

In this section, we consider the case of cubics for any value of $n$. The main result in this section extends Lemma \ref{Exr=7} and completes the case where $d = 3$.

\begin{Theorem}
	\label{5.1bis}
	Suppose that $n \ge 2$. A set of $r$ general double points in $\PP^n$ is not $\text{AH}_n(3)$ if and only if $n=4$ and $r = 7$.
\end{Theorem}

\begin{proof} The case where $n = 2$ was already proved in Section \ref{sec.P2}. The case of $n = 4$ has been discussed in Lemma \ref{Exr=7}. For $n\geq 3$ and $n\neq 4$ we proceed by considering two possibilities depending on the congruence of $n$ modulo 3.
	
\textsc{Case 1:} $n\equiv 0,1$ (mod 3). In these cases $(n+2)(n+3)$ is a multiple of 6. Thus, $\frac{1}{n+1}{n+3 \choose 3} = \frac{(n+2)(n+3)}{6} \in \ZZ$ and by Remark \ref{rmk.Hsemi}, it suffices to show that a set of $r = \frac{(n+2)(n+3)}{6}$ general double points in $\PP^n$ is $\text{AH}_n(3)$.

We shall use induction on $n$ to show that the ideal of $r$ general double points in $\PP^n$ contains no cubics. The first base case, when $n \equiv 0$ (mod 3), is $n = 3$. By Remark \ref{rmk.Hsemi}, the assertion amounts to showing that a set $X$ of 5 general double points in $\PP^3$ is $\text{AH}_3(3)$, ie., its defining ideal contains no cubics. 
Without loss of generality we may write $X=Y\cup \{Q\}$ where $Q=[1:1:1:1]$, $Y=\{P_0,P_1,P_2,P_3\}$ and $P_i=[e_i]=[0:\ldots:1:0\ldots:0]$ for $0\leq i \leq 3$. Then $Y$ is a star configuration of 4 points, and a basis of $[I_Y^{(2)}]_3$ is $\{x_0x_1x_2,x_0x_1x_3,x_0x_2x_3,x_1x_2x_3\}$ (see, e.g., \cite[Cor.~3.8]{GGSV}). So any cubic $f$ in $I_X^{(2)}\subseteq I_Y^{(2)}$ is a linear combination of these basis elements. It is easily seen that imposing that the partial derivatives $(\partial/\partial x_i) f(Q)=0$ forces $f=0$.

The other base case, when $n \equiv 1$ (mod 3), is $n = 7$ and $r=15$. This can be computed directly (and verified via Macaulay 2 \cite{M2} computations).

Suppose now that $n \ge 6$ and $n \not= 7$. The inductive hypothesis applies to $n_1=n-3$. So we let $r_1$ be the integer obtained by replacing $n$ by $n_1=n-3$ in the formula for $r$, i.e. $r_1 = \frac{n(n-1)}{6}$. Let $L$ be a  codimension 3 linear subspace in $\PP^n$, after possibly a change of variables we may assume the defining ideal of $L$ is $\pp_L = (x_{n-2},x_{n-1},x_n)$. Let $X$ be a set of $r_1$ general double points in $L$ together with $r-r_1 = n+1$ general double points outside of $L$. By the semi-continuity of Hilbert function (Remark \ref{rmk.Hsemi}), it is enough to show that $I_X$ contains no cubics. Consider a point $Q$ in the support of $X$ that lies in $L$, and let $\q$ be its defining ideal. Clearly, $\q \supseteq \pp_L$. Thus, we can write $\q = \overline{\q} + \pp_L$, where $\overline{\q}$ is a linear prime in $R_1 = \CC[x_0, \dots, x_{n-3}] \simeq R/P_L$. It follows from \cite[Theorem 3.4]{HNTT} that
$$\q^{(2)} = \overline{\q}^{(2)} + \overline{\q}\cdot \pp_L + \pp_L^{(2)}.$$
Particularly, it implies that $\q^{(2)} + \pp_L = \overline{\q}^{(2)} + \pp_L$ is the defining ideal of the double point $2Q$ in $L$. Thus, by letting $\overline{X}$ be the set of $r_1$ general double points of $X$ in $L$, considered as a subscheme of $L \simeq \PP^{n-3}$, we obtain
$$I_X + \pp_L \subseteq I_{\overline{X}} + \pp_L.$$
Moreover, by the induction hypothesis applied to $\overline{X} \subseteq L \simeq \PP^{n-3}$, we have $\left[I_{\overline{X}}\right]_3 = (0).$ Therefore, $I_X + \pp_L/\pp_L$ contains no cubics. Hence, by considering the exact sequence
$$0 \lra I_X \cap \pp_L \lra I_X \lra I_X + \pp_L/\pp_L \lra 0,$$
to prove that $I_X$ contains no cubics, it remains to show that $I_X \cap \pp_L$ contains no cubics. This is the content of \autoref{claim1} below.

\begin{Claim}
	\label{claim1}
	Suppose that $n \ge 3$ and $n \not= 4$. Let $L$ be a codimension 3 linear subspace of $\PP^n$ and let $X$ be the union of $r_1 = \frac{n(n-1)}{6}$ general double points in $L$ and $n+1$ general double points outside of $L$. Then, $I_X \cap \pp_L$ contains no cubics.
\end{Claim}

\begin{proof}[Proof of \autoref{claim1}] We use also induction on $n$ to prove the assertion. The base case $n=3$ holds because, by the above, $I_X$ contains no cubics. The other base case $n = 7$ can be verified directly, or by Macaulay 2 \cite{M2} computations. Assume that $n \ge 6$.
	For the inductive step, let $M$ be a codimension 3 linear subspace of $\PP^n$ such that $L\cap M$ has codimension 6 in $\PP^n$ (any general codimension 3 linear subspace would work). Let $\pp_M$ be the defining ideal of $M$. We specialize to the following situation:
	\begin{itemize}
		\item $r_2:=\frac{(n-3)(n-3-1)}{6}=\frac{(n-3)(n-4)}{6}$ of the points of $X$ in $L$ are  general double points in $L \cap M$;
		\item the $r_1-r_2=n-2$ remaining points of $X$ in $L$ lie outside $M$;
		\item $n-2$ of the $n+1$ points of $X$ lying outside of $L$ are general double points in $M$;
		\item and the last 3 points of $X$ outside of $L$ are general double points outside $L\cup M$.
	\end{itemize}
	
	By the semi-continuity of Hilbert function, it suffices to show that $I_X \cap \pp_L$ contains no cubics in this particular case. From the short exact sequence
		$$
	0 \lra I_X\cap \pp_L \cap \pp_M \lra I_X\cap \pp_L \lra (I_X\cap \pp_L) + \pp_M/\pp_M\lra 0,
	$$
	it suffices to prove the other two terms of this exact sequence contain no cubics. As before, observe that
	$$
	(I_X \cap \pp_L) + \pp_M \subseteq (I_{\ol{X}} \cap \pp_{\ol{L}}) + \pp_M,
	$$
	where $\ol{X}$ denotes the set of points of $X$ lying in $M \simeq \PP^{n-3}$, and $\ol{L}$ denotes the codimension 3 subspace $L \cap M$ of $M \simeq \PP^{n-3}$. As above, it can be seen that, in $M$, $\ol{X}$ is the union of $r_2$ general double points lying in $\ol{L}$ and $n-2$ general double points outside of $\ol{L}$. Thus, by the induction hypothesis, the ideal $(I_{\ol{X}} \cap \pp_{\ol{L}}) + \pp_M/\pp_M$ of $R/\pp_M \simeq \CC[y_0,\ldots,y_{n-3}]$ contains no cubics. Hence, it remains to show that $I_X \cap \pp_L \cap \pp_M$ contains no cubics. This follows from \autoref{claim2} below.
\end{proof}

\begin{Claim}
	\label{claim2}
	Suppose that $n\geq 3$ and $n\neq 4$. Let $L,M$ be two general codimension 3 linear subspaces of $\PP^n$. Let $X \subseteq \PP^n$ be the union of $r_2 = \frac{(n-3)(n-4)}{6}$ general double points in $L \cap M$, $n-2$ general double points in $L \setminus M$, $n-2$ general double points in $M \setminus L$, and 3 general double points outside of $L \cup M$. Then, $I_X \cap \pp_L \cap \pp_M$ contains no cubics.
\end{Claim}

\begin{proof}[Proof of \autoref{claim2}] Let $Z$ be the set of double points obtained by removing the $r_2$ double points in $L \cap M$ from $X$. Clearly, $I_Z \supseteq I_X$. We shall prove the stronger statement that $I_Z \cap \pp_L \cap \pp_M$ contains no cubics. The statement for $n = 3, 5, 6$, and $7$ can be verified by direct computations (e.g. using Macaulay 2 \cite{M2}). We shall use induction to prove the statement for $n \ge 8$.
	
Let $N$ be another general codimension 3 linear subspace of $\PP^n$ and let $\pp_N$ be its defining ideal. We specialize the points as follow: we take $n-5$ of the $n-2$ double points of $Z$ lying in $L$ to be in $L\cap N$, we take $n-5$ of the $n-2$ double points of $Z$ lying in $M$ to be in $M\cap N$, and we take the 3 general double points of $Z$ outside of $L \cup M$ to be in $N \simeq \PP^{n-3}$. By the semi-continuity of Hilbert function, it suffices to show that for this special configuration of $Z$, the ideal $I_Z$ contains no cubics.

Consider the following short exact sequence
	$$
	0 \lra I_Z \cap \pp_L \cap \pp_M \cap \pp_N \lra I_Z \cap \pp_L \cap \pp_M \lra I_Z \cap \pp_L \cap \pp_M + \pp_N/\pp_N \lra 0.
	$$
By an argument similar to the proof of \autoref{claim1}, we have $I_Z \cap \pp_L \cap \pp_M + \pp_N \subseteq I_{\overline{Z}} \cap \pp_{\overline{L}} \cap \pp_{\overline{M}} + \pp_N$, where $\overline{\bullet}$ represents the restrictions of $\bullet$ 
to $N \simeq \PP^{n-3}$. The induction hypothesis applies to $\overline{Z}$, so  $I_{\overline{Z}} \cap \pp_{\overline{L}} \cap \pp_{\overline{M}} + \pp_N/\pp_N$ contains no cubics. Therefore, to establish the desired statement, it remains to show that $I_Z \cap \pp_L \cap \pp_M \cap \pp_N$ contains no cubics. This follows from \autoref{claim3} below, noting that $n-5 \ge 3$.
\end{proof}

\begin{Claim}
	\label{claim3}
	Suppose that $n \ge 5$. Let $L$, $M$, and $N$ be general codimension 3 linear subspaces of $\PP^n$. Let $X \subseteq \PP^n$ be the union of 3 general double points in $L \setminus (M \cup N)$, 3 general double points in $M \setminus (L \cup N)$, and 3 general double points in $N \setminus (L \cup M)$. Then $I_X \cap \pp_L \cap \pp_M \cap \pp_N$ contains no cubics.
\end{Claim}

\begin{proof}[Proof of \autoref{claim3}] Direct computations (e.g. via Macaulay 2 \cite{M2})  verify the statement for $n = 5$ and $n = 6$. (Notice that in \cite[Prop.~5.2]{BO} it is incorrectly stated that when $n=6$ the ideal $\pp_L \cap \pp_M\cap \pp_N$ contains no quadrics.) Assume that $n \ge 7$.
Without loss of generality, we may assume that $\pp_L=(x_0,x_1,x_2)$ and $\pp_M=(x_3,x_4,x_5)$, so $\pp_L \cap \pp_M=\pp_L \pp_M$; in particular, $\pp_L \cap \pp_M$ is minimally generated by 9 quadrics, so $H_{R/\pp_L \cap \pp_M}(2)=H_R(2) - 9$

 Let $\kappa:=\pp_L \cap \pp_M \cap \pp_N$, so we need to show that $I_X\cap \kappa$ contains no cubics.

We shall first show that $\kappa$ contains no quadrics. Indeed, if $n \ge 8$ then we may assume that $\pp_N = (x_6, x_7, x_8)$. In this case, $\kappa = \pp_L \pp_M \pp_N$ is generated in degree 3. On the other hand, if $n = 7$ then we may assume that $\pp_N = (x_6, x_7, x_0-x_3)$.  
Now, consider the short exact sequence
$$
0 \lra R/\kappa \lra R/\pp_L \cap \pp_M  \oplus R/\pp_N \lra R/(\pp_L \cap \pp_M)+\pp_N \lra 0.
$$
Since $R/\left(\pp_L \cap \pp_M\right) + \pp_N = R/\pp_L \pp_M + \pp_N = R/(x_0,x_1,x_2)(x_3,x_4,x_5), x_0-x_3, x_6,x_7)$ is isomorphic to $B:=\CC[x_1,\dots, x_5]/(x_1,x_2,x_3)(x_3,x_4,x_5)$, then we have 
$$H_{R/\kappa}(2) = H_{R/\pp_L \cap \pp_M}(2) + H_{R/\pp_N}(2) - H_B(2) = (H_R(2)-9) + 15 - H_B(2)=27+15 - H_B(2).$$
Since $B$ contains all the quadrics in $\CC[x_1, \dots, x_5]$ except for the 9 generators of the ideal  $(x_1,x_2,x_3)(x_3,x_4,x_5)$, then $H_B(2) = 15 - 9 = 6$. Therefore, $H_{R/\kappa}(2) = 42-6=36=H_R(2)$, showing that $[\kappa]_2=0$.

Now, by the above short exact sequence, since $\dim R \geq 5$ one has $\depth R/\kappa \ge 2$. Let $h$ be a general linear form in $R$ and let $H$ be the hyperplane in $\PP^n$ defined by $h$; since $\depth R/\kappa \geq2$, we may assume $h$ is regular on $R/\kappa$. Let $\overline{R} = R/(h)$ and $\overline{\kappa}$ be the image of $\kappa$ in $\overline{R}$. From the standard short exact sequence
$$
0 \lra R/\kappa \lra R/\kappa \lra \ovl{R}/\ovl{\kappa} \lra 0
$$ one obtains that $\depth \ovl{R}/\ovl{\kappa} \geq 1$, i.e. $\overline{\kappa}$ is saturated in $\overline{R}$.

We now specialize the configuration so that all 9 double points of $X$ are on the hyperplane $H \simeq \PP^{n-1}$ and let $I = I_X \cap \kappa$ for simplicity of notation. Consider the short exact sequence
$$0 \lra (I:h)(-1) \lra I \lra (I,h)/(h) \lra 0.$$
Since the points in $X$ are lying on $H$, we have $I:h = \kappa:h = \kappa$. Thus, this sequence can be rewritten as
$$0\lra \kappa(-1) \lra I \lra (I,h)/(h) \lra 0.$$
As we have shown, $\kappa$ has no quadrics, so $\kappa(-1)$ has no cubics. Hence, to show that $I$ contains no cubics, it remains to show that the image $\overline{I}$ of $I$ in $\overline{R}$ has no cubics. This is indeed true by induction on $n$, since $\overline{I} \subseteq (\overline{I})^{\text{sat}}$ and 
$(\overline{I})^{\text{sat}}$ is the defining ideal of $X$ in $H \simeq \PP^{n-1}$.
\end{proof}

\textsc{Case 2:} $n \equiv 2$ (mod 3). In this case, $\frac{{n+3 \choose 3}}{n+1}=\frac{n^2+5n+6}{6} = \frac{(n+1)(n+4)}{6} + \frac{1}{3}$, and since $n\equiv 2$ (mod 3), we know $\frac{(n+1)(n+4)}{6}$ is an integer. So, we let $r_0 = \frac{(n+2)(n+3)}{6} - \frac{1}{3} = \frac{(n+1)(n+4)}{6}$ and set $\delta = {n+3 \choose 3} - r_0=\frac{n+1}{3}$. By Remark \ref{rmk.Hsemi}, to prove the desired statement, it suffices to show that sets of $r=r_0 \text{ and } r_0+1$  general double points are  $\text{AH}_n(3)$. To this end, it is enough to show that a scheme $X \subseteq \PP^n$ consisting of $r_0$ general double points and a general subscheme $\eta$ supported at another general point with degree $\delta$ is $\text{AH}_n(3)$. Indeed, it is easy to see that $X$ has multiplicity exactly ${n+3 \choose 3}$. Thus, by a proof similar to the one of Lemma \ref{reduce}, it can be shown that if $X$ is $\text{AH}_n(3)$ then so is a set of $r_0$ general double points in $\PP^n$. On the other hand, a set of $r_0+1$ general double points contains $X$ as a subscheme, so its Hilbert function in degree $d$ is at least that of $X$, which is ${n+3 \choose 3}$, i.e. it is already maximal. Particularly, a set of $r_0+1$ general double points also has maximal Hilbert function in degree 3.

As in Case 1, we shall use induction on $n\geq 2$ to show that $X$ is $\text{AH}_n(3)$. 
The case $n=2$ is proved in Theorem \ref{P2}. The induction step proceeds along the same lines as Case 1. The only difference is at \autoref{claim1}, which shall be replaced by the following

\begin{Claim}
	\label{claim4}
	Suppose that $n \ge 2$. Let $L$ be a general codimension 3 linear subspace in $\PP^n$. Let $X \subseteq \PP^n$ the union of $r_1' = \frac{(n-2)(n+1)}{6}$ general double points in $L$, $(n+1)$ general double points outside of $L$, and a general subscheme $\eta$ supported at a point $Q \in L$ and of multiplicity $\delta$ such that $\eta \cap L$ has multiplicity $\delta-1 = \frac{n-2}{3}$. Then, $I_X \cap \pp_L$ contains no cubics.
\end{Claim}

\begin{proof}[Proof of \autoref{claim4}] One proceeds by induction exactly as in the proof of \autoref{claim1}.
\end{proof}
The proof of Theorem \ref{5.1bis} is now completed.
\end{proof}


\section{Open problems} \label{sec.open}

In this section we discuss a few open problems. Let us state clearly that there are many other interesting questions outside the ones that we include here.  For instance, as indicated by Appendix \ref{app.secant} below, the polynomial interpolation is closely connected to secant varieties and Waring rank. Thus, there are many other problems and questions that are of interest to researchers working in these areas or studying, for example, containment problems for ordinary and symbolic powers of ideals, other interpolation problems and invariants associated to symbolic powers of ideals. 

However, to keep this section aligned with the other sections, we restrict ourselves to problems and questions 
{\em related to the Alexander--Hirschowitz theorem}. It is implicit that this small set of problems and conjectures is far from being comprehensive, and it should be considered as a sample -- aimed at young researchers -- of the many problems in this active area of research.

We begin by observing that Theorem \ref{AH} describes the Hilbert function of $I_Y^{(2)}$ for every set $Y$ of {\em general} points in $\PP^n$ with a finite list of exceptions (the Hilbert functions in these cases can be worked out individually). A starting point is asking for a characterization of  the Hilbert function of $I_Y^{(2)}$ for any set of points $Y$ in $\PP^n$. 

To state this general problem, for $n\geq 1$ and $r\geq 1$, let $\mathcal H_n(r)$ be the set of all Hilbert functions $H_{R/I_Y^{(2)}}$ where $Y$ is a set of $r$ points in $\PP^n$.
\begin{Problem}\label{prob.HF1}
Characterize the numerical functions which are Hilbert functions of $I_Y^{(2)}$ for some set $Y$ of points in $\PP^n$, i.e. for every $n\geq 1$ characterize all elements in
$$
\mathcal H_n :=\bigcup_{r\geq 1}\mathcal H_n(r) = \left\{  H_{R/I_Y^{(2)}} \,\mid\, Y \text{ is a set of points in }\PP^n \right\}.
$$
\end{Problem}

In this generality, so far this has been a very challenging problem, see, for instance, the surveys of Gimigliano \cite{Gim} and Harbourne \cite{Harbourne}. 
Since Problem \ref{prob.HF1} is easy for points in $\PP^1$ (see Proposition \ref{P1}), and, to the best of our knowledge, it is still open in $\PP^2$ (see \cite{GMS2006} and \cite{GHM2009} for some recent work in this direction), then one might attempt to tackle this first nontrivial case: 

\begin{Problem}\label{prob.HF1a}
Characterize the numerical functions which are Hilbert functions of $I_Y^{(2)}$ for some set $Y$ of points in $\PP^2$, i.e.  characterize all elements in
$$
\mathcal H_2:=\left\{  H_{R/I_Y^{(2)}} \,\mid\, Y \text{ is a set of points in }\PP^2 \right\}.
$$
\end{Problem}

In investigating a family of Hilbert functions, it is natural to determine the existence of ``minimal'' and ``maximal'' elements. In fact, we can 
define a partial order on $\mathcal H_n(r)$ by setting $$H_{R/I_Y^{(2)}}\leq H_{R/I_Z^{(2)}} \qquad \text{ if }\qquad H_{R/I_Y^{(2)}}(d)\leq H_{R/I_Z^{(2)}}(d) \text{ for every }d\geq 1.$$ 
Notice that every $H\in \mathcal H_n(r)$ satisfies
$$
H(d) \leq \min\left\{{n+d \choose d}, r(n+1)\right\}
$$
and, by Theorem \ref{AH}, equality holds for any general set of points (with a few exceptions). Therefore, Theorem \ref{AH} in particular proves the existence of maximal elements in $\mathcal H_{n}(r)$ (with a few exceptions), and numerically characterizes what these maximal Hilbert functions are. It is a natural problem to determine the potential existence and characterization of {\em minimal} elements of $\mathcal{H}_n(r)$.

\begin{Problem}\label{prob.HF1b}
Fix $n,r\geq 1$.
\begin{itemize}
\item[$($a$)$] Prove the existence of a minimal element in $\mathcal H_n(r)$.
\item[$($b$)$] Determine the minimal element in $\mathcal H_n(r)$.
\end{itemize}
\end{Problem}
A partial answer to Problem \ref{prob.HF1b} was given for double points in $\PP^2$ in \cite{GMS2006, GH2011}, where the problem is solved when $r={t \choose 2}$ or $r\leq 11$. 
Another natural approach in examining the Hilbert function of double points is to specify that the points are \emph{lying on a given subscheme}, e.g. on a rational normal curve or a conic.

\begin{Problem}\label{prob.HF2}
For $n\geq 1$, let $C_n$ be the rational normal curve in $\PP^n$. For any $r\geq 1$, determine the Hilbert function of $R/I_Y^{(2)}$ where $Y$ is a set of $r$ general points on $C_n$.
\end{Problem}
If the rational normal curve $C_n$ is replaced by a conic then Problem \ref{prob.HF2} has a satisfactory answer, given by Geramita, Harbourne and Migliore \cite{GHM2009}. 

Another problem along the lines of the Alexander--Hirschowitz theorem is to determine the Hilbert functions of sets of general double points in {\em multiprojective spaces}. In general, however, points in multiprojective spaces are harder to understand than points in projective spaces. (e.g., a set of points in $\PP^{n_1} \times \dots \times \PP^{n_k}$ does not need to be Cohen--Macaulay.) Much work has been put forward to understand, in general, numerical invariants and properties of points in  the first nontrivial case of a multiprojective space, i.e., $\PP^1 \times \PP^1$, (see, e.g., \cite{GVT2015}). 

While the Hilbert function for a general set of double points in $\PP^1 \times \PP^1$ is known (see \cite{VT2005}), that for an {\em arbitrary} set of double points in $\PP^1 \times \PP^1$ is not yet completely classified.

\begin{Problem}\label{prob.HF3}
Let $R=\CC[x_0,\ldots,x_3]$ and fix any $r\geq 1$. Determine the possible Hilbert functions of $R/I_Y^{(2)}$ where $Y$ is any set of $r$ points in $\PP^1 \times \PP^1$.
\end{Problem}

We observe, in passing, that similarly to how the Alexander-Hirschowitz theorem is closely related to the study of secant varieties of Veronese embeddings of $\PP^n$, Problem \ref{prob.HF3} is intimately connected to the study of secant varieties of Segre-Veronese varieties (cf. \cite{CGG2005}).
\bigskip

In general, understanding the symbolic square $I_Y^{(2)}$ of a set $Y$ of simple points is far from being a completed task. Since for certain questions $I_Y^2$ is more understood than $I_Y^{(2)}$, a possible approach is to compare $I_Y^{(2)}$ and $I_Y^2$, or simply to consider the module $I_Y^{(2)}/I_Y^2$. 

For instance, Galetto, Geramita, Shin and Van Tuyl \cite{GGSV} defined a first possible measure aimed at quantifying the gap between the $m$-th symbolic power of an ideal and the $m$-th ordinary power. They dubbed this measure the {\em $m$-th symbolic defect of an ideal $J$}, and they defined it to be
$$
{\rm sdef}(J,m):=\mu(J^{(m)}/J^m).
$$
(Here, $\mu(M)$ denotes the minimal number of generators of a finitely generated $R$-module $M$.) The problem of determining symbolic defects of an ideal is open, even for the defining ideal of a general set of points.

\begin{Problem}\label{prob.HF4a}
	Compute ${\rm sdef}(I_Y,2)$ for any set $Y$ of general simple points in $\PP^n$.
\end{Problem}
Problem \ref{prob.HF4a} seems to be open even in $\PP^2$.

\begin{Problem}\label{prob.HF4b}
	Compute ${\rm sdef}(I_Y,2)$ for any set $Y$ of general simple points in $\PP^2$.
\end{Problem}

A first partial result towards Problem \ref{prob.HF4b} is \cite[Thm~6.3]{GGSV}, where the authors determined the second symbolic defect when $|Y|\leq 9$ and $|Y|\neq 6$. These are precisely the set of points whose second symbolic defect is either 0 or 1. They also proved that if $|Y|=6$ of $|Y|>10$, then ${\rm sdef}(I_Y,2)>1$, however, the precise value is not known.

Inspired by studies on symbolic defects of an ideal, we can consider a similar invariant defined by examining the Hilbert function instead of the minimum number of generators. Particularly, for $m \in \NN$, define the \emph{$m$-th symbolic HF-defect} of an ideal $J$ to be the Hilbert function of $J^{(m)}/J^m$, i.e.
$$
{\rm sHFdef}(J,m):=H_{J^{(m)}/J^m}.
$$

\begin{Problem}\label{prob.HF4}
Compute ${\rm sHFdef}(I_Y,2)$ for any set $Y$ of general points in $\PP^n$.
Equivalently, compute the Hilbert function $H_{R/I_Y^2}$ for any set $Y$ of general points in $\PP^n$.
\end{Problem}
The equivalence of the statements given in Problem \ref{prob.HF4} follows because $H_{I_Y^{(2)}/I_Y^2} = H_{R/I_Y^2} - H_{R/I_Y^{(2)}}$, and by Theorem \ref{AH} we already know $H_{R/I_Y^{(2)}}$.

Most of the above problems are aimed at understanding symbolic squares of ideals of points; however, the most natural, important and challenging question raised by Theorem \ref{AH} is to prove an analogue of Theorem \ref{AH} for any symbolic power of any ideal defining a set of general points in $\PP^n$.

\begin{Problem}\label{prob.HF5}
Let $n\geq 1$ and $R=\CC[x_0,\ldots,x_n]$. For every fixed $m\geq 3$, determine the Hilbert function of $R/I_Y^{(m)}$ for a set $Y$ of general points in $\PP^n$.
\end{Problem}

Problem \ref{prob.HF5} is one of the main open problems in interpolation theory. Even the case where $m=3$ is still wide open.
\begin{Problem}\label{prob.HF5a}
Let $n\geq 1$ and $R=\CC[x_0,\ldots,x_n]$. Determine the Hilbert function of $R/I_Y^{(3)}$ for a set $Y$ of general points in $\PP^n$.
\end{Problem}

As we have seen in Theorem \ref{AH}, one expects to have a finite list of exceptional cases, for which the general statement does not hold. A starting point toward Problem \ref{prob.HF5a} is to determine a similar list of exceptional cases for triple general points.

\begin{Problem}\label{prob.HF5b}
Let $n\geq 1$ and $R=\CC[x_0,\ldots,x_n]$. Determine all the potential exceptional cases for Problem \ref{prob.HF5a}, i.e., find a finite list $\mathcal L$ such that if $3Y$, for a general set of points $Y \subseteq \PP^n$, is not $\text{AH}_n(d)$ then $Y\in \mathcal L$.
\end{Problem}

A well-known conjecture, often referred to as the \emph{SHGH Conjecture},  raised (and refined) over the years by Segre, Harbourne, Gimigliano and Hirschowitz, provides the first step toward a solution to Problem \ref{prob.HF5b} by predicting what these exceptional cases are expected to be. We shall state a special case of this conjecture, namely, the uniform points in $\PP^2$. See, for instance, \cite{CHHVT} for a more general statement and details on the SHGH Conjecture. 

An irreducible homogeneous polynomial $F \in R = \CC[x,y,z]$ is said to be \emph{exceptional} for a set $Y = \{P_1, \dots, P_r\}$ of points in $\PP^2$ if
$$\deg(F)^2 - \sum_{i=1}^r n_i^2 = -3\deg(F) + \sum_{i=1}^r n_i = -1,$$
where $n_i$ is the highest vanishing order of $F$ at $P_i$, for $i = 1, \dots, r$, i.e. $n_i=\max\{t\in \NN_0\,\mid\, F\in \pp_i^t\}$ (and $\pp_i$ is the defining ideal of $P_i$)

\begin{Conjecture}\label{SHGH}(SHGH Conjecture)
Let $Y$ be a general set of points in $\PP^2$ and let $m \in \NN$. Then, $mY$ is not $\text{AH}_n(d)$ if and only if there exists an irreducible homogeneous polynomial $F \in R$ that is exceptional for $Y$ such that $F^s$, for some $s > 1$, divides every homogeneous polynomial of degree $d$ in $I_Y^{(m)}$.
\end{Conjecture}

The ultimate goal naturally would be to determine the Hilbert function of every non-uniform symbolic power of any set $Y$ of general points in $\PP^n$ (i.e., the Hilbert function of $\p_1^{m_1} \cap \ldots \cap \p_r^{m_r}$,  where $Y=\{P_1,\dots,P_r\}$ is a set of general points in $\PP^n$ and $\p_i$ is the defining ideal of $P_i$ for every $i$).

Harbourne \cite{Ha05} showed that this problem would be solved if one is able to determine $\alpha(\p_1^{m_1} \cap \ldots \cap \p_r^{m_r})$ for every choice of the multiplicities $m_i\in \ZZ_+$. Here, for any homogeneous ideal $J$,
$$
\alpha(J):=\min\left\{d\geq 0 \,\mid\, [J]_d\neq 0\right\},$$
is the {\em initial degree} of $J$.
Hence, the problem of determining the initial degree of symbolic powers of ideals of points would solve the ultimate problem on interpolation. However, as one may expect, determining $\alpha$ is usually very challenging, even in the uniform case and even for points in $\PP^2$. For instance, the following celebrated conjecture of Nagata, which arose from his work on Hilbert's 14-th problem \cite{Nagata}, remains open.

\begin{Conjecture}\label{Nagata}(Nagata's Conjecture)
For any set $Y$ of $r\geq 10$ general points in $\PP^2$, and any $m\geq 2$ one has
$$
\alpha(I_Y^{(m)}) > m\,\sqrt{r}.
$$
\end{Conjecture}

Conjecture \ref{Nagata} was proved by Nagata when $r$ is a perfect square. A large body of literature is dedicated to this conjecture (cf. \cite{Harbourne2} and references therein and thereafter). Connections have also been made with other problems, for instance symplectic packing problems (see, e.g., \cite[Section 5]{Bi}). Nevertheless, the conjecture still seems out of reach at the moment. The interested reader will find in the literature many variations and different viewpoints on Nagata's Conjecture, e.g., in \cite{CHMR}.

Given the difficulty in establishing the bound predicted by Nagata's conjecture, it is natural to ask for weaker bounds. In this direction, we mention that for an arbitrary set of points $Y$ in $\PP^n$, a weaker bound for $\alpha(I_Y^{(m)})$ was formulated by G. V. Chudnovsky in \cite{Chud}.

\begin{Conjecture}[Chudnovsky] \label{Chud}
Let $Y$ be an arbitrary set of points in $\PP^n$. For every $m\geq 1$, we have
$$
\frac{\alpha(I_Y^{(m)})}{m} \geq \frac{\alpha(I_Y) + n - 1}{n}.
$$
\end{Conjecture}

Chudnovsky's conjecture has been established for 
\begin{itemize}
\item points in $\PP^2$ (see \cite{Chud, HH-2011}),
\item \emph{general} points in $\PP^3$ (see \cite{D-2012}),
\item points on a quadric (see \cite{FMX}),
\item \emph{very general} points in $\PP^n$ (in \cite{DTG-2017} for large number of points, and in \cite{FMX} for any number of points),
\item large number of general points in $\PP^n$ (see \cite{BGHN}).
\end{itemize}

Very recently, in a personal communication with the authors, R. Lazarsfeld suggested a geometric intuitive evidence for why one may expect the existence of counterexamples to Conjecture \ref{Chud}. Thus, instead of trying to prove Conjecture \ref{Chud}, one may look for counterexamples. It should be noted that partial results stated earlier suggest that a potential counterexample should be a special configuration and have high singularity outside of the given set of points.


\begin{appendices}

\section{Appendix: Secant varieties and the Waring problem} \label{app.secant}

In this section, we briefly describe the connection between the polynomial interpolation problem, particularly the Alexander--Hirschowitz theorem, and studies on secant varieties and the big Waring problem for forms. 

Throughout this section, let $V$ be a vector space of dimension $(n+1)$ over $\CC$. Then $\PP^n$ can also be viewed as $\PP(V)$, the projective space of lines going through the origin in $V$. For $f \in V \setminus \{0\}$, let $[f]$ denote the line spanned by $f$ in $V$ and, at the same time, the corresponding point in $\PP(V)$.

Let $S$ be the symmetric algebra of $V$. Then $S$ is naturally a graded algebra, given by $S = \bigoplus_{d \ge 0} S^dV$, where the $d$-th symmetric tensor $S^dV$ is a $\CC$-vector space of dimension ${n+d \choose n}$. Note that the dual $R = S^*$ is the polynomial ring $\CC[x_0, \dots, x_n]$ identified as the coordinate ring of $\PP(V)$. 

\begin{Definition}
	\label{def.Veronese}
	Let $V$ be a $(n+1)$-dimensional vector space over $\CC$.
	\begin{enumerate}
		\item The \emph{$d$-th Veronese embedding} of $\PP(V)$ is the map $\nu_d: \PP(V) \rightarrow \PP(S^dV)$, given by
		$$[v] \mapsto [v^d] = [\underbrace{v \otimes \dots \otimes v}_{d \text{ times}}].$$
		Equivalently, $\nu_d$ is the map $\PP^n \rightarrow \PP^N$, where $N = {n+d \choose d}-1$, defined by
		$$[a_0: \dots:a_n] \mapsto [a_0^d: a_0^{d-1}a_1: \dots: a_n^d],$$
		where the coordinates on the right are given by all monomials of degree $d$ in the $a_i$'s.
		\item The \emph{$d$-th Veronese variety} of $\PP(V)$, denoted by $V^n_d$, is defined to be the image $\nu_d(\PP(V))$.
	\end{enumerate}
\end{Definition}

\begin{Lemma}
	\label{rmk.tangentspace}
	Let $f \in V \setminus \{0\}$. The tangent space $T_{[f^d]}(V^n_d)$ of $V^n_d$ at the point $[f^d]$ is spanned by
	$$\{[f^{d-1}g] \in \PP(S^dV) ~\big|~ g \in V\}.$$
\end{Lemma}

\begin{proof} Let $g \in V \setminus \{0\}$. The line $\ell$ passing through $[f] \in \PP(V)$, whose tangent vector at $[f]$ is given by $[g]$, is parameterized by $\epsilon \mapsto [f + \epsilon g]$. The image of this line via the Veronese embedding $\nu_d$ is given by $\epsilon \mapsto [(f+\epsilon g)^d]$. As $f^d$ corresponds to the value $\epsilon=0$, then the tangent vector of $\nu_d(\ell)$ at $\nu_d([f])$ is
	$$\big[\dfrac{d}{d\epsilon}\Big|_{\epsilon = 0} (f+\epsilon g)^d\big] = [df^{d-1}g] = [f^{d-1}g].$$
The statement then follows.
\end{proof}

\begin{Lemma}
	\label{lem.TSobservation}
	Let $f \in V \setminus \{0\}$.
	\begin{enumerate}
		\item There is a one-to-one correspondence between hyperplanes in $\PP(S^dV)$ containing $[f^d]$ and hypersurfaces of degree $d$ in $\PP(V)$ containing $[f]$.
		\item There is a one-to-one correspondence between hyperplanes in $\PP^(S^dV)$ containing $T_{[f^d]}(V^n_d)$ and hypersurfaces of degree $d$ in $\PP(V)$ singular at $[f]$.
	\end{enumerate}
\end{Lemma}

\begin{proof} (1) Let $z_0, \dots, z_N$, where $N = {n+d \choose d}-1$ be the homogeneous coordinates of $\PP(S^dV)$. The equation for a hyperplane $H$ in $\PP(S^dV)$ has the form
	$$a_0z_0 + \dots + a_Nz_N = 0.$$
	By replacing $z_i$ with the corresponding monomial of degree $d$ in the $x_i$'s, this equation gives a degree $d$ equation that describes a degree $d$ hypersurface in $\PP(V)$. Clearly, this hypersurface is the preimage $\nu_d^{-1}(H)$ of $H$. Furthermore, since $\nu_d^{-1}([f^d]) = [f]$, if $H$ passes through $[f^d]$ then $\nu_d^{-1}(H)$ contains $[f]$.
	
	(2) Let $\{e_0, \dots, e_n\}$ be a basis of $V$ whose dual basis in $R$ is $\{x_0, \dots, x_n\}$. By a linear change of variables, we may assume that $f = e_0$. That is, $[f] = [1:0:\dots:0] \in \PP^n_K$. Then, the defining ideal of $[f]$ is $(x_1, \dots, x_n)$.
	
	It follows from Lemma \ref{rmk.tangentspace} that $T_{[f^d]}(V^n_d)$ is spanned by
	$$\{[e_0^d], [e_0^{d-1}e_1], \dots, [e_0^{d-1}e_n]\}.$$
	As before, the equation for a hyperplane $H$ in $\PP(S^dV)$ has the form
	$a_0z_0 + \dots + a_Nz_N = 0.$
	By using lexicographic order, we may assume that $z_0, \dots, z_n$ are variables corresponding to monomials $x_0^d, x_0^{d-1}x_1, \dots$, $x_0^{d-1}x_n$. Then, $H$ contains $T_{[f^d]}(V^n_d)$ if and only if $a_0 = \dots = a_n = 0$. It follows that the equation for $\nu_d^{-1}(H)$ is a linear combination of monomials of degree $d$ not in the set $\{x_0^d, x_0^{d-1}x_1, \dots, x_0^{d-1}x_n\}$. Particularly, these monomials have degree at least 2 in the variables $x_1, \dots, x_n$. Hence, $\nu_d^{-1}(H)$ is singular at $[f]$.
\end{proof}

We obtain an immediate corollary.

\begin{Corollary}
	\label{cor.TS0}
	Let $X = \{[f_1], \dots, [f_k]\} \subseteq \PP(V)$ be a set of points. Let $\m_1, \dots, \m_k$ be the defining ideal of these points. Then, there is a bijection between the vector space of degree $d$ elements in $\bigcap_{i=1}^k \m_i^2$ and the vector space of hyperplanes in $\PP(S^dV)$ containing the linear span of $T_{[f_1^d]}(V^n_d), \dots, T_{[f_k^d]}(V^n_d)$.
\end{Corollary}

\begin{proof} The conclusion follows from Lemma \ref{lem.TSobservation}, noticing that $\m_i^2$ is the ideal of polynomials in $R$ singular at $[f_i] \in \PP(V)$ for all $i = 1, \dots, k$.
\end{proof}

Recall that if $I$ is any (homogeneous) ideal in $R$, then $[I]_d=I \cap R_d$ is the vector space of all homogeneous equations of degree $d$ in $I$.

\begin{Corollary}
	\label{cor.TS}
	Let $X = \{[f_1], \dots, [f_k]\} \subseteq \PP(V)$ be a set of points. Let $\m_1, \dots, \m_k$ be the defining ideals of the points in $X$. Then,
	$$\dim_\CC \big[\bigcap_{i=1}^k \m_i^2\big]_d = N - \dim \langle T_{[f_1^d]}(V^n_d), \dots, T_{[f_k^d]}(V^n_d)\rangle.$$
\end{Corollary}

\begin{proof} The conclusion is an immediate consequence of Corollary \ref{cor.TS0} and basic linear algebra facts.
\end{proof}

\begin{Definition}
	\label{def.secant}
	Let $X$  be a projective variety. For any nonnegative integer $r$, the \emph{$r$-secant variety} of $X$, denoted by $\sigma_r(X)$, is defined to be
	$$\sigma_r(X) = \overline{\bigcup_{P_1, \dots, P_r \in X} \langle P_1, \dots, P_r\rangle}^{\text{ Zariski closure}}.$$
\end{Definition}

Note that $\sigma_r(V^n_d)$ is an irreducible variety for all $r$.

\begin{Remark}
	\label{rmk.expecteddim}
	Let $X \subseteq \PP^N$ be a projective scheme of dimension $n$. Then,
	$$\dim \sigma_r(X) \le \min \{rn+r-1, N\} = \min\{(n+1)r-1, N\}.$$
	When the equality holds we say that $\sigma_r(X)$ has \emph{expected dimension}.
\end{Remark}

\begin{Lemma}[First Terracini's Lemma]
	\label{lem.Terracini1}
	Let $Y \subseteq \PP^n$ be a projective scheme. Let $p_1, \dots, p_r$ be general points in $Y$. Let $z \in \langle p_1, \dots, p_r\rangle$ be a general point in the linear span of $p_1, \dots, p_r$. Then,
	$$T_z(\sigma_r(Y)) = \langle T_{p_1}(Y), \dots, T_{p_r}(Y)\rangle.$$
\end{Lemma}

\begin{proof} Let $Y(\tau) = Y(\tau_1, \dots, \tau_n)$ be a local parametrization of $Y$. Let $Y_j(\tau)$ represent the partial derivative with respect to $\tau_j$, for $j = 1, \dots, n$. Suppose that $p_i$ corresponds to $\tau^i = (\tau^i_1, \dots, \tau^i_n)$ in this local parametrization.
	
By definition, $T_{p_i}(Y)$ is spanned by the tangent vectors $Y(\tau^i) + \epsilon Y_j(\tau^i)$, for $j = 1, \dots, n$. Thus, $\langle T_{p_1} (Y), \dots, T_{p_r} (Y)\rangle$ is the affine span of $\{Y(\tau^i), Y_j(\tau^i) ~\big|~ i = 1, \dots, r, j = 1, \dots, n\}$.

On the other hand, a general point $z$ in $\sigma_r(Y)$ is parametrized by $Y(\tau^k) + \sum_{i=1}^{r-1} \gamma_i Y(\tau_i)$. By considering partial derivatives at $z$, it can be seen that $T_z(\sigma_r(Y))$ is also the affine span of $\{Y(\tau^i), Y_j(\tau^i) ~\big|~ i = 1, \dots, r, j = 1, \dots, n\}$. The lemma is proved.
\end{proof}

The following theorem establishes the equivalence between being $\text{AH}_n(d)$ for double points and having expected dimension for secant varieties.

\begin{Theorem}
	\label{thm.equivSecant}
	A set of $r$ general double points in $\PP^n$ is $\text{AH}_{n}(d)$ if and only if $\sigma_r(V^n_d)$ has expected dimension.
\end{Theorem}

\begin{proof}
	Let $X = \{q_1, \dots, q_r\}$ be a set of $r$ general simple points in $\PP^n$, and let $\m_1, \dots, \m_r$ be their defining ideals. Set $p_i = \nu_d(q_i)$ for $i = 1, \dots, r$. By Corollary \ref{cor.TS}, we have
	$$\dim_\CC \big[\bigcap_{i=1}^r \m_i^2\big]_d = N - \dim \langle T_{p_1}(V^n_d), \dots, T_{p_r}(V^n_d)\rangle.$$
	It follows that
	\begin{align*}
	h_{\PP^n}(2X,d) & = {n+d \choose d} - \dim_\CC \big[\bigcap_{i=1}^r \m_i^2\big]_d \\
	& = \dim \langle T_{p_1}(V^n_d), \dots, T_{p_r}(V^n_d)\rangle + 1.
	\end{align*}
	By the genericity assumption of the points and the fact that $\sigma_r(V^n_d)$ is an irreducible variety, Lemma \ref{lem.Terracini1} now gives
	$$h_{\PP^n}(2X,d) = \dim \sigma_r(V^n_d) + 1.$$
	The conclusion now follows, noting that $h_{\PP^n}(2X,d) \le \min\{(n+1)r, {n+d \choose n}\}$ and $\dim \sigma_r(V^n_d) \le \min\{(n+1)r-1, {n+d \choose n}-1\}.$
\end{proof}

Via its connection to secant varieties of Veronese varieties, the interpolation problem is also closely related to the big Waring problem for forms.

\begin{Definition}
	\label{def.Waringrk}
	Let $F \in R$ be a homogeneous polynomial of degree $d$. The \emph{Waring rank} of $F$, denoted by $\rk(F)$, is defined to be the \emph{minimum} $s$ such that
	$$F = \ell_1^d + \dots + \ell_s^d$$
	for some linear forms $\ell_1, \dots, \ell_s \in R_1$.
\end{Definition}

The Waring problems for forms ask for bounds or precise values for the Waring rank of homogeneous polynomials.

\begin{Definition}
	\label{def.G}
	Let $n,d$ be positive integers and $R=\CC[x_0,\ldots,x_n]$.
	\begin{enumerate}
		\item Set $G(n,d) := \min\{ s \in \NN ~\big|~ \rk(F) \le s \text{ for a general element } F \in R_d\}.$
		\item Set $g(n,d) := \min\{ s \in \NN ~\big|~ \rk(F) \le s \text{ for any element } F \in R_d\}.$
	\end{enumerate}
\end{Definition}

The \emph{Big Waring Problem} and \emph{Little Waring Problem}, respectively, are to determine $G(n,d)$ and $g(n,d)$. It is easy to see that $g(n,d) = \max\{ \rk(F) ~\big|~ F \in R_d\}.$ The connection between $G(n,d)$ and secant varieties comes from the following result.

\begin{Lemma}
	\label{lem.Waring}
	$G(n,d) = \min\{ r ~\big|~ \sigma_r(V^n_d) = \PP(S^dV)\}.$
\end{Lemma}

\begin{proof} Fix a basis $\{e_0, \dots, e_n\}$ of $V$ whose dual basis in $R$ is $\{x_0, \dots, x_n\}$. Let $\theta: R \rightarrow S$ be the natural isomorphism defined by $x_i \mapsto e_i$. Consider a form $F \in R_d$. By definition, $\rk(F) \le r$ if and only if there exist linear forms $\ell_1, \dots, \ell_r \in R_1$ such that
	\begin{align}
	F = \ell_1^d + \dots + \ell_r^d \label{eq.powersum}
	\end{align}
	Let $h_i' = \theta(\ell_i)$ for $i = 1, \dots, r$. Then, (\ref{eq.powersum}) holds if and only if
	$\theta(F) = h_1^d + \dots + h_r^d$.
	By scalar scaling if necessary, this is the case if and only if $[\theta(F)] = \langle [h_1^d], \dots, [h_r^d]\rangle$.
	
	By the definition of $\sigma_r(V^n_d)$ (being the Zariski closure of the union of secant linear subspaces), it then follows that $\sigma_r(V^n_d) = \PP(S^dV)$ if and only if $\bigcup_{P_1, \dots, P_r \in V^n_d} \langle P_1, \dots, P_r \rangle$ contains a general point of $\PP(S^dV)$. Hence, $\rk(F) \le r$ for a general element $F \in R_d$ if and only if $\sigma_r(V^n_d) = \PP(S^dV)$.
\end{proof}


\section{Appendix: Symbolic powers}\label{app.Symbolic}

Considering the considerable literature on symbolic powers of ideals, we have included in this appendix only a minimal amount of definitions and results. We refer the interested reader to the recent, comprehensive survey on the subject \cite{DDGHN18}.

\begin{Definition}
Let $R=\CC[x_0,\ldots,x_n]$ and let  $I$ be an ideal with no embedded associated primes. For every $m\in \ZZ_+$, the {\em $m$-th symbolic power of $I$} is the $R$-ideal
$$
I^{(m)} = \bigcap_{p\in \Ass(R/I)}\left(I^mR_p\cap R\right).
$$
Additionally, one sets $I^{(0)}=R$.
\end{Definition}

For every integer $m\geq 2$ and ideal $I$, one has $I^m\subseteq I^{(m)}$ and in general this is a strict inclusion. A notable exception is when $I$ is a complete intersection, in which case, $I^m=I^{(m)}$ for every $m\geq 1$.

The following result gives a way to compute symbolic powers of ideals of points.
\begin{Proposition}
Let $X=\{P_1,\ldots,P_r\}$ be a set of simple points in $\PP^n$, i.e. its defining ideal $I_X=\p_1\cap \ldots \cap \p_r$ in $R=\CC[x_0,\ldots,x_n]$ is a radical ideal. Then for every $m\geq 1$
$$
I_X^{(m)} = \p_1^m \cap  \dots \cap \p_r^m.
$$
\end{Proposition}

\begin{Definition}
Let $P$ be a point in $\PP^n$ with defining ideal $\p$. For $m\geq 1$, we write $mP$ for the subscheme of $\PP^n$ with defining ideal $\p^{m}$. One often calls $mP$ a {\em fat point} subscheme of $\PP^n$.

If $P_1,\ldots,P_r$ are points in $\PP^n$ with defining ideals $\p_1, \ldots, \p_r$, then the defining ideal of the fat point scheme $X=m_1P_1+m_2P_2+\ldots + m_rP_r$ is $I_X:=\p_1^{m_1} \cap \ldots \cap \p_r^{m_r}$.
\end{Definition}

\begin{Example}\label{E3pts}(The defining ideal of 3 non-collinear double points in $\PP^2$)
Let $X=\{[1:0:0], [0:1:0], [0:0:1]\}\subseteq \PP^2$ be the three coordinate points in $\PP^2$, and let $I_X=(x_1,x_2) \cap (x_0,x_2)\cap (x_0,x_1)=(x_0x_1,x_0x_2,x_1x_2)\subseteq \CC[x_0,x_1,x_2]$ be its defining ideal. The defining ideal of $mX$ is
$$
I_X^{(m)}=(x_1,x_2)^m \cap (x_0,x_2)^m \cap (x_0,x_1)^m.
$$
For instance, $2X$ is defined by $I_X^{(2)}=(x_1,x_2)^2 \cap (x_0,x_2)^2 \cap (x_0,x_1)^2$, and by computing this intersection one obtains
$$
I_X^{(2)}= (x_0x_1x_2)+I_X^2.
$$
In particular, $I_X^2\neq I_X^{(2)}$. (In fact, more generally, $x_0^tx_1^tx_2^t\in I_X^{(2t)}$ for all $t\in \ZZ_+$.)
\end{Example}

An important theorem proved by Zariski \cite{Z49} and Nagata \cite{N62} (and generalized by Eisenbud and Hochster \cite{EH1979}) provides a first illustration of the geometric relevance of symbolic powers of ideals: they consist of all hypersurfaces vanishing with order at least $m$ on the variety defined by $I$. 

\begin{Theorem}[Zariski-Nagata] \label{ZariskiNagata}
Let $I$ be a radical ideal in $R = \CC[x_0,\ldots,x_n]$ and let $s \ge \NN$. Then,
$$
\begin{array}{ll}
I^{(s)} & = {\displaystyle \bigcap_{\m \in \Max(R),\,I\subseteq \m} \m^s.}\\
	& = \{f\in R\,\mid\, \text{ all partial derivatives of }f \text{ of order }\leq s-1 \text{ lie in }I\}.
	\end{array}
$$
\end{Theorem}

\begin{Example}
Let $X$, $I_X$ and $R$ be as in Example \ref{E3pts}. We have claimed that $x_0x_1x_2 \in I_X^{(2)}$. One can check this easily using Zariski--Nagata theorem. Each of the partial derivatives of $x_0x_1x_2$ with respect to one of the variables is a minimal generator of $I_X$. Thus, all partial derivatives of order at most $1$ of $x_0x_1x_2$ lie in $I_X$, and so by Zariski--Nagata theorem, $x_0x_1x_2\in I_X^{(2)}$.

More geometrically, $V(x_0x_1x_2)$ is the union of the three lines $V(x_0)\cup V(x_1)\cup V(x_2)$. Each of the points at the intersection of two of the three lines are singular points. Since these three intersections are the points of $X$, it follows by Zariski--Nagata theorem that $x_0x_1x_2\in I_X^{(2)}$.
\end{Example}


\section{Appendix: Hilbert function} \label{app.Hilbertfn}

In this section, we record some basic properties of Hilbert functions, especially relative to sets of points.
We start with a simple lemma potentially allowing the use of Linear Algebra to investigate interpolation problems.

\begin{Lemma}\label{lem.count}
Let $P\in \PP^n$ be a point with defining ideal $\p \subseteq R = \CC[x_0, \ldots, x_n]$. Then
$[\p^m]_d$ consists of all solutions of a homogeneous linear system of $\binom{n+m-1}{n}$ equations in $\binom{n+d}{n}$ variables.
In particular, the rank of this linear system is $H_{R/\p^m}(d)$.
\end{Lemma}

\begin{proof}
Let $F$ be a generic homogeneous equation of degree $d$ in $n+1$ variables, i.e.
$$
F = \sum_{M\in T_d}c_M M\in \CC[\{c_M\}, x_0,\ldots,x_n]
$$
where $T_d$ consists of all the $\binom{n+d}{d}$ monomials of degree $d$ in $R$.

By Zariski--Nagata's theorem, the equation $F$ vanishes with multiplicity $m$ at a point $P\in \PP^n$ $\Llra$ all the $(m-1)$-order (divided power) derivatives of $F$ vanish at $P$.

Now, consider any $(m-1)$-th partial order derivative of $F$ with respect to the $x_i$'s and substitute the coordinates of $P$ in for the variables. We obtain a linear combination of the coefficients $c_M$'s; this linear combination is zero if and only if that partial derivative of $F$ vanish at $P$.
Therefore, there is a bijective correspondence between the solutions to the system of these $\binom{n+m-1}{n}$ linear equations in the unknowns $c_M$'s and all hypersurfaces of degree $d$ passing through $P$ at least $m$ times. This proves the statement.
\end{proof}

The Hilbert function of a graded ring counts the number of linearly independent forms in a given degree.
\begin{Definition}
Let $R=\CC[x_0,\ldots,x_n]$, and let $M=\bigoplus_{i\geq 0}M_i$ be a graded $R$-module (e.g., $M=R/I$ where $I$ is a homogeneous ideal). Then, $M_i$ is a $\CC$-vector space for every $i \ge 0$. The {\em Hilbert function of $M$} is the function $H_M : \ZZ_{\ge 0} \rightarrow \NN$, given by
$$
H_M(d) :=\dim_\CC M_d.
$$
For any $a\in \ZZ$, $M(a)$ is defined as the graded $R$-module whose degree $j$ component is $[M(a)]_j=M_{a+j}$. 
\end{Definition}

In general, for $d$ large, the Hilbert function of $M$ agrees with a polynomial of degree $\dim(M)-1$, which is called the {\em Hilbert polynomial} of $M$. Its normalized leading coefficient is an integer $e(M)$ called the {\em multiplicity of $M$}. When $M$ is 1-dimensional and Cohen-Macaulay, $H_M(d)$ is non-decreasing and eventually equals the multiplicity of $M$. We recapture this property in the following proposition.

\begin{Proposition}\label{e(M)}
Let $R = \CC[x_0,\ldots,x_n]$ and let $M$ be a Cohen-Macaulay graded $R$-module with $\dim(M)=1$. Then $H_{M}(d-1)\leq H_{M}(d)$ for all $d \in \NN,$ and $H_{M}(d)=e(M)$ for $d\gg 0$.
In particular, $H_{M}(d)\leq e(M)$ for every $d \in \NN$.
\end{Proposition}

\begin{proof}
Since $\dim(M)=1$, then the Hilbert polynomial of $M$ is just the constant function $e(M)$, so $H_{M}(d)=e(M)$ for $d\gg 0$.

Since $M$ is Cohen-Macaulay there exists a linear form $x\in R$ that is regular on $M$. Let $\ovl{R}=R/(x)$ be the Artinian reduction of $R$, and let $\ovl{M}=M/(x)M$. The short exact sequence
$$
0 \lra M(-1) \stackrel{\cdot x}{\lra } M \lra \ovl{M} \lra 0
$$
and the additivity of Hilbert function under short exact sequence yield $H_{\ovl{M} }(d) = H_{M}(d) - H_{M}(d-1)$. As $H_{\ovl{M} }(d)$ is of course non-negative for every $d \in \NN$, then $H_{M}(d-1)\leq H_{M}(d)$ for all $d \in \NN$.
\end{proof}

To complement the previous result, one can use the so-called Associativity Formula for $e(R/I)$ to prove the following statement computing the multiplicity of any set of fat points in $\PP^n$.

\begin{Proposition}\label{e(R/I)2}
Let $Y=\{P_1,\ldots,P_r\}$ be a set of points in $\PP^n$, and let $X=\{m_1P_1, \ldots, m_rP_r\}$. Then
$$
e(R/I_X) = \sum_{i=1}^r \binom{n+m_i-1}{n}.
$$
\end{Proposition}

By counting equations and variables, one immediately obtains an upper bound for the Hilbert function of any ideal associated to (possibly fat) points.

\begin{Corollary}\label{cor.count2}
Let $X=\{m_1P_1,\ldots,m_1P_r\}$ be a set of points in $\PP^n$ with multiplicities $m_1,\dots,m_r$. Then
$$
H_{R/I_X}(d) \leq \min\left\{\binom{d+n}{n}, \sum_{i=1}^r \binom{n+m_i-1}{n}\right \},
$$
or, equivalently, $H_{I_X}(d)\geq \max\left\{0, \binom{d+n}{n} - \sum_{i=1}^r \binom{n+m_i-1}{n}\right \}$.
\end{Corollary}

\begin{proof}
By Propositions \ref{e(M)} and \ref{e(R/I)2} we have $H_{R/I_X}(d) \leq e(R/I_X) = \sum_{i=1}^r \binom{n+m_i-1}{n}.$
One also has $H_{R/I_X}(d)\leq H_{R}(d) = \binom{d+n}{n}$.
\end{proof}

When the equality in Corollary \ref{cor.count2} is achieved, we obtain the definition of $\text{AH}_n(d)$, or maximal Hilbert function in degree $d$. In other papers, this property is often referred to as {\em $X$ imposes independent conditions on degree $d$ hypersurfaces in $\PP^n$}.

\begin{Definition} \label{def.AH}
	Let $X=\{m_1P_1, \ldots, m_rP_r\}\subseteq \PP^n$ be a set of $r$ points with multiplicities $m_1,m_2,\ldots,m_r$. We say that $X$ is {\em $\text{AH}_n(d)$} (or has {\em maximal Hilbert function in degree $d$}), if
	$$
	H_{R/I_X}(d) = \min\left\{\binom{d+n}{n}, \sum_{i=1}^r \binom{n+m_i-1}{n}\right \}.
	$$
	The number $\min\left\{\binom{d+n}{n}, \sum_{i=1}^r \binom{n+m_i-1}{n}\right \}$ is called {\em  the expected codimension in degree $d$} (for $r$ general points in $\PP^n$ with multiplicities $m_1,\ldots,m_r$).
\end{Definition}

The easiest situation in Definition \ref{def.AH} is when $m_1 = \dots = m_r = 1$, i.e., the points in $X$ are all simple points.

\begin{Theorem}\label{simplegen}
A set $X$ of $r$ general simple points in $\PP^n$ is $\text{AH}_n(d)$ for every $d\geq 1$. 
\end{Theorem}

\begin{proof} The statement follows from a simple observation, via Lemma \ref{lem.count}, that the condition $H_{R/I_X}(d) = \min\{{n+d \choose d}, r\}$ is an open condition.
\end{proof}

\begin{Example}
	Let $X=\{P_1,2P_2,4P_3\}\subseteq \PP^3$, then $e(R/I_X)=25$, and
	$X$ is $\text{AH}_3(d)$ if and only if
	$$
	H_{R/I_X}(d) = \min\left\{\binom{d+3}{3}, 1 + 4 + 20\right \}
	$$
	i.e., if its Hilbert function is $H_{R/I_X}=(1,4,10,20,25,25,25,\ldots)$.
\end{Example}

Other simple situations, where we can quickly prove that the property $\text{AH}_n(d)$ holds, are when $X$ is supported at a single point, i.e. $r=1$, or when $d = 1$, or $n = 1$. 

\begin{Remark}\label{r=1}
	A single double point $X=\{2P\}$ in $\PP^n$ is $\text{AH}_n(d)$ for all $n$ and $d$.
\end{Remark}

\begin{proof}After a change of coordinates we may assume that $I_P=(x_1,\ldots,x_n)$, so $I_X = I_P^{(2)}=I_P^2$.
	If $d=1$ then $[I_P^{(2)}]_1=0$, so $H_{R/I_P^{(2)}}(1) = n+1 = \min\{n+1, n+1\}$.
	If $d\geq 2$ then $[R/I_P^{(2)}]_d = \langle x_0^d, x_0^{d-1}x_1,\ldots,x_0^{d-1}x_n\rangle$. Thus,
	$$
	H_{R/I_P^{(2)}}(d) = n+1 = \min\left\{\binom{n+d}{d}, n+1\right\},
	$$
and the statement follows.
\end{proof}

\begin{Remark}\label{d=1}
	Any set $2Y$ of $r$ double points (not necessarily general) in $\PP^n$ is $\text{AH}_n(1)$.
\end{Remark}

\begin{proof}
	For any $r\geq 1$, one needs to show that
	$$
	H_{I_Y^{(2)}}(1)  =\max\{0, \binom{n+1}{n} - r(n+1)\}=0.
	$$
	Let $P\in Y$ be any point, then $[I_Y^{(2)}]_1\subseteq  [I_P^{(2)}]_1=[I_P^2]_1 = (0)$. Thus, $I_Y^{(2)}$ contains no linear forms.
\end{proof}

\begin{Proposition}\label{P1}
	Let $Y$ be a set of $r$ distinct simple points in $\PP^1$. Then $mY$ is $\text{AH}_1(d)$ for every $d,m\in\ZZ_+$.
\end{Proposition}

\begin{proof}
	Notice that the defining ideal of any point in $\PP^1$ is just a principal prime ideal generated by a linear form. Thus, if $Y=\{P_1,\ldots,P_r\}$ then $I_Y$ is a principal ideal generated by a form of degree $r$. It follows that $I_Y^{(m)}=I_Y^m\subseteq R=\CC[x_0,x_1]$ is a principal ideal of degree $rm$, so $I_Y^{(m)}\cong R(-rm)$.  Thus, $H_{I_Y^{(m)}}(d)$ is
	$$\max\left\{0, \binom{d+1-rm}{1}\right\} = \max\left\{0, d+1-rm \right\}=\max\left\{0, \binom{d+1}{d} - rm\right\}$$
	Therefore, $mY$ is $\text{AH}_1(d)$.
\end{proof}

The following lemma allows us to restrict attention to only a finite number of values of $r$ in proving the Alexander-Hirschowitz theorem.

\begin{Lemma}\label{reduce}
	Let $X$ be a set of $r$ points in $\PP^n$, with multiplicities $m_1,\ldots,m_r$, which is $\text{AH}_n(d)$.
	\begin{enumerate}
		\item If $H_{R/I_X}(d) =\binom{d+n}{n}$, then $X'$ is also $\text{AH}_n(d)$ for any larger set  $X'\supseteq X$ consisting of $r'\geq r$ points of $X$ with multiplicities $m_1'\geq m_1, \dots, m_{r'}'\geq m_{r'}$.
		\item If $H_{R/I_X}(d) =\sum_{i=1}^r \binom{n+m_i-1}{n}=e(R/I_X)$, then $X'$ is also $\text{AH}_n(d)$ for any subset $X'\subseteq X$ consisting of $r'\leq r$ points of $X$ with multiplicities $m_1'\leq m_1, \dots, m_{r'}'\leq m_{r'}$. 
	\end{enumerate}	
\end{Lemma}

\begin{proof}
	(1) Since $X\subseteq X'$ then $I_{X'}\subseteq I_X$. By assumption $[I_X]_d=0$, thus also $[I_{X'}]_d=0$, which implies that $H_{R/I_{X'}}(d) =\binom{d+n}{n}$.
	
	(2) Since $X'$ is a subscheme of $X$ and $X$ is multiplicity $d$-independent, the assertion is a direct consequence of Lemma \ref{multind}.
\end{proof}

We conclude this appendix by stating a numerical characterization of when it is possible to find simple points to be added to a given scheme in order to change its Hilbert function by a prescribed value. 

\begin{Proposition}[\protect{\cite[Lemma~3]{Ch00}}] \label{prop.hyperp}
Let $I$ be a saturated homogeneous ideal in $R=\CC[x_0,\ldots,x_n]$, and let $\ell$ be a linear form that is regular on $R/I$. TFAE:
\begin{enumerate}
\item There exists a set $Y_0$ of $u$ points in $V(\ell)$ such that
$$
H_{R/(I\cap I_{Y_0})}(t) = H_{R/I}(t) + u.
$$
\item $H_{R/I}(t) +u \leq H_{R/I}(t-1) + \binom{n+t-1}{t}.$
\end{enumerate}
\end{Proposition}


\section{Appendix: Semi-continuity of the Hilbert function and reduction to special configurations}\label{app.semi-cont}

The starting point of the proof of Theorem \ref{AH} is the observation that to establish the $\text{AH}_n(d)$ property for a general set of double points, in non-exceptional cases, we only need to exhibit a specific collection of double points with the $\text{AH}_n(d)$ property. This is because Hilbert functions have the so-called \emph{lower semi-continuity} property. This is the content of this appendix.

We begin by defining \emph{generic} points and the \emph{specialization} of points. Throughout this appendix, we shall fix a pair of positive intgers $n$ and $r$. Recall that $R = \CC[x_0, \dots, x_n]$ is the homogeneous coordinate ring of $\PP^n$. Let $\ul{z} = \{z_{ij} ~\big|~ 1 \le i \le r, 0 \le j \le n\}$ be a collection of $r(n+1)$ indeterminates, and let $\CC(\ul{z})$ be the purely transcendental field extension of $\CC$ be adjoining the variables in $\ul{z}$. Let $S = \CC(\ul{z})[x_0, \dots, x_n]$ be the homogeneous coordinate ring of $\PP^n_{\CC(\ul{z})}$.

\begin{Set-up}\label{setup2} \quad
\begin{enumerate} 
\item By the \emph{generic set of $r$ points}, we mean the set $Z = \{Q_1, \dots, Q_r\}$, where $Q_i = [z_{i0}: \dots: z_{in}]$, for $i = 1, \dots, r$, are points with the generic coordinates in $\PP^n_{\CC(\underline{z})}$. Let $I_Z \subseteq S$ denote the defining ideal of $Z$.
\item Let $\underline{\lambda}=(\lambda_{ij}) \in \mathbb{A}_{\CC}^{r(n+1)}$ be such that for each $i = 1, \dots, r$, $\lambda_{ij} \neq 0$ for some $j$. Define the set $Z(\underline{\lambda}) = \{Q_1(\underline{\lambda}), \dots, Q_r(\underline{\lambda})\}$ of points in $\PP^n$, with $Q_i(\underline{\lambda}) = [\lambda_{i0}: \dots: \lambda_{in}]$.
Let $I_{\underline{\lambda}} \subseteq R$ be the defining ideal of $Z(\underline{\lambda})$.
\end{enumerate}
We call $Z(\ul{\lambda})$ the {\em specialization of the generic points at $\ul{\lambda}$}, and call $I_{\ul{\lambda}}$ the {\em specialization of the ideal $I_Z$ at $\ul{\lambda}$}.
\end{Set-up}

To define precisely the notions of general points and very general points one often employs Chow varieties. However, one can also use dense Zariski-open subsets of $A^{r(n+1)}$ (see, e.g., \cite[Lemma~2.3]{FMX}) for these purposes, and this is the point of view we take.

\begin{Definition}\label{DefGeneral}
One says that a property $\mathcal P$
\begin{itemize}
\item {\em holds for a general set of $r$ points} of $\PP_{\CC}^n$ if there is a dense Zariski-open subset $U\subseteq \mathbb A_{\CC}^{r(n+1)}$ such that $\mathcal P$ holds for $Z(\ul{\lambda})$ for all $\ul{\lambda}\in U$;
\item  {\em holds for a very general set of $r$ points} of $\PP_{\CC}^n$ if $\mathcal P$ holds for $Z(\ul{\lambda})$ for all $\ul{\lambda}\in U$ where $U$ is an an intersection of countably many dense Zariski-open subsets of $A_{\CC}^{r(n+1)}$.
\end{itemize}
\end{Definition}

The lower semi-continuity of Hilbert functions that we shall use is stated in the following theorem.

\begin{Theorem}[Lower-semi-continuity of the Hilbert function]
	\label{specializations}
Assume Set-up~\ref{setup2}. Then, for any $m, d \in \NN$, we have
$$H_{I_Z^{(m)}}(d) \leq H_{I_Y^{(m)}}(d)$$
for any set $Y$ of $r$ points in $\PP^n$. Moreover, for fixed $m\geq 1$ and $d\geq 1$, the equality $H_{I_Z^{(m)}}(d) = H_{I_Y^{(m)}}(d)$ holds for a general set of points $Y \subseteq \PP^n$.
\end{Theorem}

\begin{proof} Note that every set $Y$ of $r$ points in $\PP^n$ can be viewed as a specialization $Z(\ul{\lambda})$ of the generic set of $r$ points.
The proof is similar to the proof of \cite[Thm~2.4]{FMX}.
For every $s\geq 1$, set
$$
W_s:=\{\ul{\lambda} \in \mathbb{A}_{\CC}^{r(n+1)} \,\mid\, H_{I_{\ul{\lambda}}^{(m)}}(d) \geq s\}.
$$
We claim that $W_s$ is a Zariski-closed subset of $\mathbb{A}_{\CC}^{r(n+1)}$ for any $s\geq 1$.

To see it, let $f=\sum_{|\alpha|=d}C_{\alpha}\ul{x}^{\alpha}\in R[C_{\alpha}]$ be a generic homogeneous polynomial of degree $d$, where $\ul{x}^{\alpha}$ are the monomials of degree $d$ in $R$. Let $\partial_{\ul{\beta}}\ul{x}^{\alpha}$ denote the partial derivative of $\ul{x}^{\alpha}$ with respect to $\ul{\beta}$.

Now, let $\mathbb D_{m,d}$ be the matrix with columns indexed by all monomials in $R_d$, rows indexed by all partial derivatives $\ul{\beta}$ with $|\ul{\beta}|\leq m-1$, and whose rows are
$$\left[\partial_{\ul{\beta}}x_{0}^d \quad\ldots\quad
\partial_{\ul{\beta}}\ul{z_i}^{\ul{\alpha}}\quad\ldots\quad \partial_{\ul{\beta}} x_{n}^d \right].$$

Let $[\mathbb B_{m,d}]_{\ul{\lambda}}$ be the $r$ by $1$ block matrix
$$
\mathbb B_{m,d} = \left[ \begin{array}{c}
\mathbb D_{m,d}(P_1)\\
\mathbb D_{m,d}(P_2)\\
\vdots \\
\mathbb D_{m,d}(P_k)\\
\end{array}
\right]
$$
where $\mathbb D_{m,d}(P_1)$ is the specialization of the matrix $\mathbb D_{m,d}$ at the point $P_i$, i.e.  we replace $x_0,\ldots,x_n$ by $\lambda_{i,0},\ldots,\lambda_{i,n}$, respectively.

Then the forms $f=\sum_{|\alpha|=d}C_{\alpha}\ul{x}^{\alpha}$ of degree $d$ in $I_{\ul{\lambda}}^{(m)}$ are in a bijective correspondence with the non-trivial solutions to the
system of equations (in the variables $C_{\alpha}$)
$$
[\mathbb B_{m,d}]_{\ul{\lambda}}\cdot \left[C_{(d,\ldots, 0)}\; \ldots \;C_{\ul{\alpha}}\; \ldots \; C_{(0,\ldots, d)}\right]^T = \ul{0}.
$$
It follows that $\ul{\lambda}\in W_s$ if and only if the null-space of this linear system has dimension at least $s$, which is holds if and only if the number $r\binom{m+n}{m-1}$ of rows of $[\mathbb B_{m,d}]_{\ul{\lambda}}$ is less than $\binom{d+n}{n}-(s-1)$ or $r\binom{m+n}{m-1} \geq \binom{d+n}{n}-(s-1)$ and ${\rm rk}[\mathbb B_{m,d}]_{\ul{\lambda}}<\binom{d+n}{n}-(s-1)$. In either case we have a closed condition in $\mathbb A^{r(n+1)}$. This proves the claim.

To prove the inequality in the statement we prove that when one takes $s_0:=H_{I_Z^{(m)}}(d)$, then $W_{s_0}$ also contains a dense Zariski-open subset, thus showing that $W_{s_0}$ is the entire space.

Indeed, let $f_1,\ldots,f_{s_0}$ be linearly independent forms of degree $d$ in $I_Z^{(m)}$. We may assume that each $f_i \in \CC(\ul{z})[x_0,\ldots,x_n]$. Let $M$ be the matrix whose $i$-th row consists of the coefficients of each monomial $\ul{x}^{\alpha}$ in $f_i$. By assumption $M$ has maximal rank, i.e. $s_0$, so at least one of the minors of size $s_0$ of $M$ does not vanish. It follows that there exists a dense Zariski-open subset $\widetilde{U_{t}}$ of specializations $\ul{z}\longmapsto \ul{\lambda}$ ensuring that the specialization does not make this minor vanish, thus for any $\ul{\lambda} \in \widetilde{U}$ we have that $(f_1)_{\ul{z}\mapsto \ul{\lambda}}, (f_2)_{\ul{z}\mapsto \ul{\lambda}}, \ldots, (f_{s_0})_{\ul{z}\mapsto \ul{\lambda}}$ are $s_0$ linearly independently forms of degree $d$ in $I_{\ul{\lambda}}^{(m)}$. This concludes the proof of the first part.

The equality now follows from this last paragraph, as it is shown in there that for any $\ul{\lambda} \in \widetilde{U_d}$ one has that $s_0:=H_{I_Z^{(m)}}(d)  = H_{I_{\ul{\lambda}}^{(m)}}(d)$.

\end{proof}

We obtain the following immediate consequences of Theorem \ref{specializations}.

\begin{Corollary}\label{special}
Fix positive integers $n$, $r$, $d$ and $m$. TFAE:
\begin{enumerate}
\item There exists a set $Y$ of $r$ points in $\PP^n$ such that $mY$ is $\text{AH}_n(d)$.
\item For any set $Y$ of $r$ general points in $\PP^n$, $mY$ is $\text{AH}_n(d)$.
\end{enumerate}
\end{Corollary}

\begin{Corollary}\label{rpts}
Fix $n,d\in \ZZ_+$. Then every set $Y$ of $r$ general double points in $\PP^n$ is $\text{AH}_n(d)$ if and only if 
there exist sets of $r$ double points in $\PP^n$ which are $\text{AH}_n(d)$  for $\left\lfloor \frac{1}{n+1}\binom{d+n}{n}\right\rfloor\leq r \leq \left\lceil \frac{1}{n+1}\binom{d+n}{n}\right\rceil$.

 Similarly, if a set of $r_0$ general points is not $\text{AH}_n(d)$, then any set of $r\neq r_0$ general double points in $\PP^n$ is $\text{AH}_n(d)$ if and only if there exist sets of $r_0-1$ and $r_0+1$ double points in $\PP^n$ that are $\text{AH}_n(d)$.
 \end{Corollary}

\begin{proof} The desired statements are direct consequences of Corollary \ref{special} and Lemma \ref{reduce}.
\end{proof}

In the last part of this section we prove a semi-continuity results in the more general setting of flat families of projective schemes.

\begin{Definition}
	Let $f: X \rightarrow Y$ be a morphism of schemes, and let $\F$ be a sheaf of $\OO_X$-modules. We say that $\F$ is \emph{$f$-flat at $x \in X$} if the stalk $\F_x$, seen as an $\OO_{Y,f(x)}$-module, is flat. We say that $\F$ is \emph{$f$-flat} if it is $f$-flat at every point in $X$.
\end{Definition}

\begin{Definition}
	A \emph{family of (closed) projective schemes} $f: X \rightarrow Y$ is a morphism $f$ of (locally) Noetherian schemes which factors through a closed embedding $X \subseteq \PP^r \times Y = \PP$, for some $r$. The family is \emph{flat} if $\OO_X$ if $f$-flat.
\end{Definition}

Let $\pp$ be a point in $Y$. Let $\CC(\pp)$ be the residue field of the local ring $\OO_{Y,\pp}$. Let $X_\pp = X \times_Y \Spec(\OO_{Y,\pp})$ and let $\PP_\pp = \PP \times_Y \Spec(\OO_{Y,\pp})$. For example, if $Y = \Spec(A)$ and $X = \Proj(R/I)$, where $R = A[x_0, \dots, x_r]$ and $I \subset R$ is a homogeneous ideal, then $X_\pp = \Proj((R/I) \otimes_A \CC(\pp))$ and $\PP_\pp = \Proj(R \otimes_A \CC(\pp))$. Note that, in general, the defining ideal of $X_\pp$ in $\PP_\pp$ may not be the same as $I \otimes_A \CC(\pp)$; rather, it is the image of the canonical map $\left( I\otimes_A \CC(\pp) \rightarrow R \otimes_A \CC(\pp)\right)$.

The following result is well-known; see, for example, \cite[Theorem III.12.8]{Hartshorne}.

\begin{Theorem}
	\label{thm.semicontinuityCoh}
	Let $f: X \rightarrow Y$ be a family of projective schemes and let $\F$ be a coherent sheaf over $X$ which is also $f$-flat. Then, for each $i \ge 0$, the function $Y \rightarrow \ZZ$ defined by
	$$\pp \mapsto \dim_{\CC(\pp)}(H^i(X_\pp, \F_\pp))$$
	is upper semicontinuous on $Y$.
\end{Theorem}

\begin{Theorem}
	\label{thm.semicontinuityHF}
	Let $f: X \rightarrow Y$ be a flat family of projective schemes. Then, for any degree $d \ge 0$, the function $Y \rightarrow \ZZ$ defined by
	$$\pp \mapsto h_{\PP_\pp}(X_\pp, d)$$
	is lower semicontinuous on $Y$.
\end{Theorem}

\begin{proof} Let $\II$ be its ideal sheaf of the embedding $X \subseteq \PP$. Let $\pp \in Y$ be any point and let $A = \OO_{Y,\pp}$.
	We have a short exact sequence
	$0 \rightarrow \II \rightarrow \OO_\PP \rightarrow \OO_X \rightarrow 0.$
By tensoring with $\CC(\pp)$, we obtain the following short exact sequence
$$0 \rightarrow \II \otimes_A \CC(\pp) \rightarrow \OO_{\PP_\pp} \rightarrow \OO_{X_\pp} \rightarrow 0.$$
Particularly, this shows that $\II \otimes_A \CC(\pp)$ is the ideal sheaf of the embedding $X_\pp \subseteq \PP_\pp$. Set $\II_\pp = \II \otimes_A \CC(\pp)$. We then have
$$h_{\PP_\pp}(X_\pp, n) = h^0(\OO_{\PP_\pp}(n)) - h^0(\II_\pp(n)).$$

Observe that $\OO_X$ is $f$-flat, and so $\II$ is also $f$-flat. Therefore, by Theorem \ref{thm.semicontinuityCoh}, the function $\pp \mapsto h^0(\II_\pp(n))$ is an upper semicontinuous function on $Y$. The conclusion now follows, since $h^0(\OO_{\PP_\pp}(n))$ is constant on $Y$.
\end{proof}


\section{Appendix: Hilbert schemes of points and curvilinear subschemes} \label{app.curvilinear}

We end the paper with our last appendix giving basic definitions and properties of curvilinear subschemes that allow the deformation argument in the {\em m\'ethode d'Horace diff\'erentielle} to work.

\begin{Definition} A finite zero-dimensional scheme $Z$ is said to be \emph{curvilinear} if $Z$ locally can be embedded in a smooth curve. That is, for every point $P$ in $Z$, the dimention $T_P(Z)$ of the tangent space is at most 1. 
\end{Definition}

\begin{Lemma} \label{lem.curvilinear}
Let $Z$ be a zero-dimensional scheme supported at one point $P$. Then $Z$ is curvilinear if and only if $Z \simeq \Spec \ \CC[t]/(t^l)$, where $l$ is the degree of $Z$.
\end{Lemma}

\begin{proof} Without loss of generality, assume that $(x_1, \dots, x_n)$ are local parameters at $P$. Let $C$ be a smooth curve to which $Z$ can be embedded in. Clearly, $P \in C$. Let $I_C = (f_1, \dots, f_s)$ be the defining ideal of $C$ in $\OO_P = \CC[x_1, \dots, x_n]$ (particularly, $s \ge n-1$). Since $C$ is smooth at $P$, the Jacobian matrix of $C$ at $P$ has rank $n-1$. Thus, by a change of variables and a re-indexing, if necessary, we may further assume that $f_i = x_i + g_i$, for $i = 1, \dots, n-1$, and $g_1, \dots, g_{n-1} \in \OO_P$.
	
Let $I_Z$ be the defining ideal of $Z$ in $\OO_P$. Since $Z$ can be embedded in $C$, we have $I_C \subseteq I_Z$. Therefore, locally at $P$, $\OO_Z$ is a quotient ring of $\CC[x_n]$. It follows that, locally at $P$, $\OO_Z \cong \CC[x_n]/(x_n^l)$ for some $l$.

The converse is clear by the same arguments.
	Observe further that localizing at $P$ (a minimal prime in $\OO_Z$) does not change the multiplicity of $\OO_Z$, or equivalently, the degree of $Z$. Hence, $\deg(Z) = l$.
\end{proof}

\begin{Corollary} \label{cor.degcurvilinear}
	Let $Z$ be a curvilinear subscheme of a double point. Then the degree of $Z$ is either 1 or 2.
\end{Corollary}

\begin{proof} By Lemma \ref{lem.curvilinear}, we have $Z \cong \Spec \ \CC[t]/(t^l)$. Since $Z$ is contained in a double point, we must have $l$ is equal to 1 or 2. Hence, $\deg(Z)$ is either 1 or 2.
\end{proof}

\begin{Remark}
	Let $Z$ be a zero-dimensional scheme with irreducible components $Z_1, \dots, Z_r$. Then, $Z$ is curvilinear if and only if $Z_1, \dots, Z_r$ are curvilinear.
\end{Remark}

The next lemma gives another way of seeing curvilinear schemes.

\begin{Lemma}
	\label{lem.curvilinear2}
	Let $Z$ be a zero-dimension scheme supported at one point $P$. Then, $Z$ is curvilinear if and only if, locally at $P$, the $\OO_Z$ is generated by one element, that is, $\OO_Z = \CC[f]$ for some $f \in \OO_Z$.
\end{Lemma}

\begin{proof} By Lemma \ref{lem.curvilinear}, if $Z$ is curvilinear then, clearly, $\OO_Z$ is generated by one element. Suppose, conversely, that $\OO_Z = \CC[f]$ for some $f \in \OO_Z$. Since $Z$ is zero-dimensional, we must have $f^l = 0$ for some $l$. By taking the smallest such $l$, we then have $\OO_Z \cong \CC[t]/(t^l)$, and so $Z$ is curvilinear by Lemma \ref{lem.curvilinear}.
\end{proof}

The main result about curvilinear subschemes that we shall use is that they form an open dense subset in the Hilbert scheme of zero-dimensional subscheme of a given degree in $\PP^n$. Particularly, this allows us to take the limit of a family of curvilinear subschemes. For this, we shall need the following lemma.

\begin{Lemma}
	\label{lem.exercise}
Let $A$ be a Noetherian ring, let $B$ be a free $A$-algebra of rank $n$. Then the set 
$$\{\pp \in \Spec \ A\,\mid\, \text{ the }K(\pp)\text{-algebra }B \otimes_A K(\pp)\text{ is generated by one element}\}$$
 is an open subset $U$ of $\Spec A$. Here, $KK(\pp)$ is the residue field $A_\pp/\pp A_\pp$ at $\pp$.
\end{Lemma}

\begin{proof}Let $U:=\{\p \in \Spec(A)\,\mid\, \text{ the }K(\pp) \text{-algebra }B \otimes_A K(\pp) \text{ is generated by one element}\}$. Clearly, if $\p \in U$ and $\q\subseteq \p$ then $\q\in U$. Thus, by Nagata's topological criterion (e.g. \cite[Thm~24.2]{Mat}), to prove that $U$ is open it suffices to show that if $\p\in U$, then there exists a non-empty open subset of $V(\p)$ contained in $U$.

Write $B=A[T_1,\ldots,T_r]/J$, since $\p\in U$ then (after possibly relabelling) we may assume that there exist $a_1,\ldots,a_{r-1}\in A \setminus \p$ and $g_1,\ldots,g_{r-1}\in A[T_1,\ldots,T_r]$ with $T_i\notin {\rm supp}(g_i)$ for any $i=1,\ldots,r-1$ such that
$$
(a_1T_1+g_1,\ldots,a_{r-1}T_{r-1}+g_{r-1})\subseteq J.
$$
Clearly, for any $\q\in V(\p) \setminus [V(a_1) \cup V(a_2) \cup \ldots \cup V(a_r)]$ we have $\q\in  U$; this concludes the proof.
\end{proof}

We are now ready to state and prove the density result of curvilinear subschemes.

\begin{Proposition}\label{prop:dense}
	Let $\mathbf{H}_{l}$ denote the Hilbert scheme of zero-dimensional subscheme of degree $l$ in $\PP^n$ and let $\mathbf{H}_l^{\text{curv}}$ denote the subset of $\mathbf{H}_l$ consisting of curvilinear subschemes of degree $l$ in $\PP^n$. Then $\mathbf{H}_l^{\text{curv}}$ is an open dense subset of $\mathbf{H}_l$.
\end{Proposition}

\begin{proof} The statement follows from Lemmas \ref{lem.curvilinear2} and \ref{lem.exercise}.
\end{proof}

\end{appendices}


\end{document}